\documentclass[12pt, twoside, draft]{article}
\usepackage[cp1251]{inputenc}
\newcommand{\CopyName}{A.\ I.\ Bandura, N.\ V.\ Petrechko, O.\ B. Skaskiv} 
\newcommand{\NAME}{A.\ I.\ Bandura, N.\ V.\ Petrechko, O.\ B.\ Skaskiv} %
\newcommand{\Year}{2016} 
\newcommand{\Volume}{?} 
\newcommand{\Number}{?} 
\newcommand{\Pages}{?--?} 
\newcommand{\rightheadtext}{Boundedness of $\mathbf{L}$-index in a bidisc} 
\usepackage[english,russian,ukrainian]{babel}
\usepackage{ms_we2}
\usepackage{color}

\renewcommand{\refname}{\refnam}
\vs
\newcommand{\tit}{Analytic functions in a bidisc  of bounded $\mathbf{L}$-index in joint variables} 
\date{}
\begin{document}
\hbox to \textwidth{\footnotesize\textsc  Математичні Студії. ?.\Volume, \No \Number
\hfill
Matematychni Studii. V.\Volume, No.\Number}
\vspace{0.3in}
\textup{\scriptsize{УДК 517.553}} \vs 
\markboth{{\NAME}}{{\rightheadtext}}
\begin{center} \textsc {\CopyName} \end{center}
\begin{center} \renewcommand{\baselinestretch}{1.3}\bf {\tit} \end{center}

\vspace{20pt plus 0.5pt} {\abstract{ \noindent A.\ I.\ Bandura, N.\ V.\ Petrechko, O.\  B.\ Skaskiv 
\textit{Analytic functions in a bidisc  of bounded $\mathbf{L}$-index in joint variables}\matref \vspace{3pt} \English 

A concept of boundedness of $\mathbf{L}$-index in joint variables (see in Bordulyak M.T. {\it The space of entire in $\mathbb{C}^n$ functions of bounded $L$-index,}  Mat. Stud., \textbf{4} (1995), 53--58. (in Ukrainian)) is generalised 
for analytic in a bidisc function. 
 We proved criteria of boundedness of $\mathbf{L}$-index in joint variables which describe local behaviour of partial derivative and 
 give an estimate of maximum modulus on a skeleton of polydisc.
 Some improvements of known sufficient conditions of boundednees of $\mathbf{L}$-index in joint variables are obtained. 



}} \vsk


\renewcommand{\refname}{\refnam}
\renewcommand{\proofname}{\ifthenelse{\value{lang}=0}{Proof}{\ifthenelse{\value{lang}=2}{Доказательство}{Доведення}}}

\vskip10pt

\medskip\noindent{\bf 1. Introduction.}\
Recently authors together with M. T. Bordulyak \cite{sufjointdir} introduced a class of entire functions of bounded of $\mathbf{L}$-index in joint variables with $\mathbf{L}(z)=(l_1(z),$ $\ldots,$ $l_{n}(z)),$
$z=(z_1,z_2,\ldots,z_n)\in\mathbb{C}^n.$ It was a generalisation of previous definition with  $\mathbf{L}(z)=(l_1(z_1),$ $\ldots,$ $l_{n}(z_n)),$
supposed by M. T. Bordulyak and M. M. Sheremeta \cite{bagzmin,prostir}. 
Meanwhile there are known papers of S. N. Strochyk, M. M. Sheremeta, V. O. Kushnir \cite{strosher,kusher}, devoted to $l$-index of 
analytic in a disc or in an arbitrary domain function. Their investigations are particularized in a monograph of Sheremeta \cite{sher} where listed a full bibliography on this topic.  
However they only considered the functions of one complex variable. There are only two papers about 
analytic in some domain functions of bounded index \cite{krishna,banduraball}. 
J. Gopala Krishna and S. M. Shah \cite{krishna}  introduced an analytic in a domain (a nonempty connected open set)
$\Omega\subset \mathbb{C}^n$ $(n\in\mathbb{N})$ function of bounded index for $\alpha=(\alpha_1,
\ldots,\alpha_n)\in \mathbb{R}^n_{+}.$  If
$\mathbf{L}(z)\equiv\left(\frac{1}{\alpha_1},\ldots,\frac{1}{\alpha_n}\right)$ and $\Omega=C^n$ then a Bordulyak-Sheremeta's definition  \cite{bagzmin,prostir} 
matches with a Krishna - Shah's definition. Besides, analytic in a domain function of bounded index 
by Krishna and Shah is an entire function. It follows from necessary condition
of $l$-index boundedness for analytic in the unit disc function (\cite{sher},Th.3.3, p.71): $\int_0^rl(t)dt\to\infty$ as $r\to 1.$ 

  In other above-mentioned preprint \cite{banduraball} authors proposed a generalisation of analytic in a domain function of bounded index, which
  was introduced by J. G. Krishna and S. M. Shah. We used slice function 
  to explore properties of analytic in the unit ball functions of bounded $L$-index in direction. This approach is well studied for entire functions in 
  \cite{BandSk,monograph,bandneobm,bandsets,bandplanezeros,openproblems}. 
  
  For analytic in the unit ball functions we proved necessary and sufficient conditions of boundedness of $L$-index in direction for
  analytic functions, got sufficient
  conditions of boundedness of $L$-index in direction for analytic solutions of PDE and estimated 
  growth of the functions, etc. 
Thus, method of slices is well suited as for entire functions in $\mathbb{C}^n$ as for analytic functions in a ball. 

Besides a ball, an important geometric object in $\mathbb{C}^n$ is a polydisc.
At same time there was not flexible definition of bounded index for analytic functions of several variables by approach of M. Sheremera, 
M. Bordulyak, M. Salmassi, F. Nuray, R. Patterson, B. C. Chakraborty \cite{bagzmin,prostir,nuraypattersonexponent2015,nuraypattersonmultivalence2015,indsalmassi,Chakraborty1995,Chakraborty2000,Chakraborty2001}.  Above we noted  that analytic function of bounded index by Krishna and Shah is an entire function. 
 Thus, necessity arises \textit{to introduce and 
to study analytic in polydisc functions of bounded $\mathbf{L}$-index in joint variables.}

\medskip\noindent{\bf 2. Main definitions and notations.}\
For simplicity we consider two-dimensional complex space, i. e. $\mathbb{C}^2.$ 
This helps to distinguish main methods of investigation. 
 Indeed our results can be easy deduced for $\mathbb{C}^n.$
 
We need some standard notations. Denote $\mathbb{R}_+=[0,+\infty),$
$\mathbf{0}=(0,0)\in\mathbb{R}^2_{+},$ $\mathbf{1}=(1,1)\in\mathbb{R}^2_{+},$ 
 $R=(r_1,r_2)\in\mathbb{R}^2_+,$ $z=(z_1,z_2).$
 For $A=(a_1,a_2)\in\mathbb{R}^2,$ $B=(b_1,b_2)\in\mathbb{R}^2,$   we put
 \begin{gather*}
 AB=(a_1b_1,a_2b_2),\ \
 A/B=(a_1/b_1,a_2/b_2), \ b\not=\mathbf{0},\ \
 A^B=a_1^{b_1}a_2^{b_2},\  b\in\mathbb{Z}_{+}^2,
 \end{gather*}
 and the notation $A<B$ means that $a_j<b_j,$ $j\in\{1,2\};$ the relation $A\leq B$ is defined similarly.
For  $K=(k_1,\ldots,k_n)\in \mathbb{Z}^n_{+}$ denote
$\displaystyle\|K\|=k_1+\cdots+k_n,$\
$K!=k_1!\cdot \ldots \cdot k_n!.$

The polydisc $\{z\in\mathbb{C}^2: \ |z_j-z_j^0|<r_j, \ j=1,2\}$ is denoted by $\mathbb{D}^2(z^0,R),$ its skeleton $\{z\in\mathbb{C}^2: \ |z_j-z_j^0|=r_j, \ j=1,2\}$ is denoted by $\mathbb{T}^2(z^0,R),$ and the closed polydisc
$\{z\in\mathbb{C}^2: \ |z_j-z_j^0|\leq r_j, \ j=1,2\}$ is denoted by $\mathbb{D}^2[z^0,R],$
      $\mathbb{D}^2=\mathbb{D}^2(\mathbf{0},\mathbf{1}),$
      $\mathbb{D}=\{z\in\mathbb{C}: \ |z|<1\}.$
For $p,$ $q\in\mathbb{Z}_+$ and partial derivative of analytic in $\mathbb{D}^2$ function $F(z_1,z_2)$ we will use the notation
  $$
  F^{(p,q)}(z_1,z_2):=\frac{\partial^{p+q}F(z_1,z_2)}{\partial z_1^p\partial z_2^q}.
  $$
Let $\mathbf{L}(z)=(l_1(z_1,z_2),l_2(z_1,z_2)),$ where $l_j(z_1,z_2): \mathbb{D}^2\to \mathbb{R}_+$ is a continuous function such that 
\begin{equation*}
\forall (z_1,z_2)\in \mathbb{D}^2 \ \ l_j(z_1,z_2)>\frac{\beta}{1-|z_j|},  \ j\in\{1,2\}
\end{equation*}
where $\beta>1$ is a some constant, $\boldsymbol{\beta}=(\beta,\beta).$ Strochyk S. N., Sheremeta M. M., Kushnir V. O.   \cite{sher,strosher,kusher} imposed a similar condition for a function $l: \mathbb{D}\to \mathbb{R}_+$  and $l: G\to \mathbb{R}_+,$ where $G$ is arbitrary domain in $\mathbb{C}$.

   An analytic  function $F\colon\mathbb{D}^2\to\mathbb{C}$ is called a function of \textit{bounded $\mathbf{L}$-index (in joint variables),} if there exists $n_0\in \mathbb{Z}_+$ such that for all $(z_1,z_2)\in \mathbb{D}^2$ and for all $(p_1,p_2)$ $\in$ $\mathbb{Z}^2_+$
\begin{gather}
     \frac{1}{p_1!p_2!} \frac{|F^{(p_1,p_2)}(z_1,z_2)|}{l_1^{p_1} (z_1,z_2)l_2^{p_2} (z_1,z_2)}\leq
\max \left \{ \frac{1}{k_1!k_2!} \frac{|F^{(k_1,k_2)}(z_1,z_2)|}{l_1^{k_1} (z_1,z_2)l_2^{k_2} (z_1,z_2)}:0\leq k_1+k_2\leq n_0 \right \}.
\label{net0}
\end{gather}
 The least such integer  $n_{0}$
is called the {\it  $\mathbf{L}$-index in joint variables of the function $F(z_1,z_2)$} and is denoted by $N(F,\mathbf{L},\mathbb{D}^2)=n_0.$ 
It is an analog of definition of entire function of bounded $\mathbf{L}$-index in joint variables in $\mathbb{C}^2$ (see \cite{sufjointdir,bagzmin,prostir}).

By $Q^2(\mathbb{D}^2)$ we denote the class of functions $\mathbf{L}$,  which satisfy the condition  
$$\forall r_j\in[0,\beta], j\in\{1,2\}: \ 0<\lambda_{1,j}(R)\leq \lambda_{2,j}(R)<\infty,$$
 where
\begin{gather} \label{lam1}
\lambda_{1,j}(R)=\inf\limits_{  (z_1^0,z_2^0)\in \mathbb{D}^2} \inf \left \{
  \frac{l_j(z_1,z_2)}{l_j(z_1^0,z_2^0)}: (z_1,z_2)\in \mathbb{D}^2\left[z^0, \frac{R}{\mathbf{L}(z^0)}\right]\right\},\\
  \lambda_{2,j}(R)=\sup\limits_{  (z_1^0,z_2^0)\in \mathbb{D}^2} \sup \left \{
  \frac{l_j(z_1,z_2)}{l_j(z_1^0,z_2^0)}: (z_1,z_2)\in \mathbb{D}^2\left[z^0, \frac{R}{\mathbf{L}(z^0)}\right]\right\}, \label{lam2}\\
\frac{R}{\mathbf{L}(z^0)}:= \left(\frac{r_1}{l_1(z_1^0,z_2^0)},\frac{r_2}{l_2(z_1^0,z_2^0)}\right). \nonumber
  \end{gather}
  \begin{example}\rm 
The function $F(z_1,z_2)=\exp{ \frac{1}{(1-z_1)(1-z_2)}}$ has bounded $\mathbf{L}$-index in joint variables with 
 $\mathbf{L}(z_1,z_2)=\big(\frac{1}{(1-|z_1|)^2(1-|z_2|)},\frac{1}{(1-|z_1|)(1-|z_2|)^2}\big)$ and $N(F,\mathbf{L},\mathbb{D}^2)=0.$
\end{example}
  \noindent{\bf 3. Behaviour of derivatives of function of bounded $\mathbf{L}$-index in joint variables.}
Denote $\mathcal{B}=(0,\beta]$ and $\mathcal{B}^2=(0,\beta]\times(0,\beta],$  where $\times$ means the Cartesian product, 
\begin{theorem} \label{petr1}
   Let $\mathbf{L} \in Q^2(\mathbb{D}^2)$. An analytic in $\mathbb{D}^2$  function $F$  has bounded  $\mathbf{L}$-index in joint variables if and only if for each
  $R\in\mathcal{B}^2$ there exist $n_0\in \mathbb{Z}_+$, $p_0>0$ such that for every $z^0=(z_1^0,z_2^0) \in\mathbb{D}^2$ there exists $(k_1^0,k_2^0)\in \mathbb{Z}_+^2$, $0\leq k_1^0+k_2^0\leq n_0$, and 
 \begin{gather}
  \max\left\{ \frac{1}{k_1!k_2!} \frac{|F^{(k_1,k_2)}(z_1,z_2)|}{l_1^{k_1} (z_1,z_2) l_2^{k_2} (z_1,z_2)}:k_1+k_2\leq n_0, (z_1,z_2)\in \mathbb{D}^2\left[z^0, \frac{R}{\mathbf{L}(z^0)}\right] \right\} \leq \nonumber\\
   \leq \frac{p_0}{k_1^0!k_2^0!}  \frac{|F^{(k_1^0,k_2^0)}(z_1^0,z_2^0)|}{l_1^{k_1^0} (z_1^0,z_2^0)l_2^{k_2^0} (z_1^0,z_2^0)}.
   \label{net1}
 \end{gather}
\end{theorem}

    \begin{proof}
	Let $F$ be of bounded  $\mathbf{L}$-index in joint variables with $N=N(F,\mathbf{L},\mathbb{D}^2)<\infty.$
For every $r_j\in(0,\beta],$ $j\in\{1,2\}$ we put
$$q= q(R)= \lfloor 2(N+1)(r_1+r_2)\prod_{j=1}^{2}(\lambda_{1,j}(R))^{-N}(\lambda_{2,j}(R))^{N+1}\rfloor+1$$
where $\lfloor x \rfloor$ is the entire part of the real number $x,$ i.e. it is a floor function.
	For $p\in\{0,\ldots,q\}$ and $z^0\in\mathbb{D}^2$ we denote
\begin{gather*}
	S_p(z^0,R)=\max\left\{\frac{|F^{(k_1,k_2)}(z_1,z_2)|}{k_1!k_2!l_1^{k_1} (z_1,z_2)l_2^{k_2} (z_1,z_2)}: 0\le k_1+k_2\leq N, z\in \mathbb{D}^2 \left[z^0,\frac{pR}{q\mathbf{L}(z^0)}\right] \right\},
\\
	S^{*}_p(z^0,R)=\max\left\{\frac{|F^{(K)}(z_1,z_2)|}{k_1!k_2!l_1^{k_1} (z_1^0,z_2^0)l_2^{k_2} (z_1^0,z_2^0)}: 0\le k_1+k_2\leq N, z\in \mathbb{D}^2 \left[z^0,\frac{pR}{q\mathbf{L}(z^0)}\right] \right\}.
\end{gather*}
Using \eqref{lam1} and $\mathbb{D}^2\left[z^0,\frac{pR}{q\mathbf{L}(z^0)}\right] \subset \mathbb{D}^2 \left[z^0,\frac{R}{\mathbf{L}(z^0)}\right] $,  we have
\begin{gather*}
  S_p(z^0,R)= \nonumber \\ \!=\!\max\left\{\frac{|F^{(k_1,k_2)}(z_1,z_2)|}{k_1!k_2!l_1^{k_1} (z_1,z_2)l_2^{k_2} (z_1,z_2)}\frac{l_1^{k_1} (z_1^0,z_2^0)l_2^{k_2} (z_1^0,z_2^0)}{l_1^{k_1} (z_1^0,z_2^0)l_2^{k_2} (z_1^0,z_2^0)}: 0\!\le\!k_1+\!k_2\leq \!N,
  z\in\! \mathbb{D}^2\! \left[z^0,\frac{pR}{q\mathbf{L}(z^0)}\!\right]\! \right\}\! \leq\! \nonumber \\ \leq S^{*}_p(z^0,R)\max\left\{\frac{l_1^{N} (z_1^0,z_2^0)l_2^{N} (z_1^0,z_2^0)}{l_1^{N} (z_1,z_2)l_2^{N} (z_1,z_2)}: z\in \mathbb{D}^2 \left[z^0,\frac{pR}{q\mathbf{L}(z^0)}\right] \right\} \leq \nonumber \\ \leq S^{*}_p(z^0,R) \frac{1}{(\lambda_{1,1}(R)\lambda_{1,2}(R))^N}=S^{*}_p(z^0,R)\prod_{j=1}^{2} (\lambda_{1,j}(R))^{-N}.
\end{gather*}
and, using \eqref{lam2},  we obtain 
\begin{gather}
  S^{*}_p(z^0,R)= \nonumber \\ = \max\left\{\frac{|F^{(k_1,k_2)}(z_1,z_2)|}{k_1!k_2!l_1^{k_1} (z_1,z_2)l_2^{k_2} (z_1,z_2)}\!\frac{l_1^{k_1} (z_1,z_2)l_2^{k_2} (z_1,z_2)}{l_1^{k_1} (z_1^0,z_2^0)l_2^{k_2} (z_1^0,z_2^0)}: \!k_1+\!k_2\leq N,
  z\in \mathbb{D}^2 \left[z^0,\frac{pR}{q\mathbf{L}(z^0)}\right] \right\} \leq \nonumber \\
  \leq \max\left\{\frac{|F^{(k_1,k_2)}(z_1,z_2)|}{\!k_1!\!k_2!l_1^{k_1} (z_1,z_2)l_2^{k_2} (z_1,z_2)}\!(\lambda_{2,1}(R))^{\!k_1} \!(\lambda_{2,2}(R))^{\!k_2}\!: \!k_1+\!k_2\leq \!N,
  z\in \mathbb{D}^2 \left[z^0,\frac{pR}{q\mathbf{L}(z^0)}\right] \right\} \leq \nonumber \\ \leq S_p(z^0,R) 
  (\lambda_{2,1}(R)\lambda_{2,2}(R))^{N}=S_p(z^0,R)\prod_{j=1}^{2} (\lambda_{2,j}(R))^N.
  \label{netss2}
\end{gather}
Let $K^{(p)}=(k_1^{(p)},k_2^{(p)})$,  $k_1^{(p)}+k_2^{(p)}\leq N$ and $z^{(p)}=(z_1^{(p)},z_2^{(p)}) \in \mathbb{D}^2 \left[z^0,\frac{pR}{q\mathbf{L}(z^0)}\right]$ be such that

\begin{gather}
S_p^*(z^0,R)=\frac{|F^{(k_1^{(p)},k_2^{(p)})}(z_1^{(p)},z_2^{(p)})|}{k_1^{(p)}!k_2^{(p)}!l_1^{k_1^{(p)}} (z_1^{0},z_2^{0})l_2^{k_2^{(p)}} (z_1^0,z_2^0)}
\label{stars}
\end{gather}
Since by the maximum principle $z^{(p)}\in \mathbb{T}^2(z^0,\frac{pR}{q\mathbf{L}(z^0)}),$ we have $z^{(p)}\neq z^0.$ We choose \\
	$
	\tilde z^{(p)}_{1}=z_1^0+\frac{p-1}{p}(z_1^{(p)}-z_1^0) $  and $
\tilde z^{(p)}_{2}=z_2^0+\frac{p-1}{p}(z_2^{(p)}-z_2^0).
$
Then for every $j\in\{1,2\}$ we have that
\begin{gather}
  |\tilde z^{(p)}_{j}-z_j^0|=\frac{p-1}{p}|z^{(p)}_{j}-z_j^0|=\frac{p-1}{p} \frac{pr_j}{ql_j(z_1^0,z_2^0)}
  \label{zet}\\
  |\tilde z^{(p)}_{j}-z_j^{(p)}|=|z_j^{0}+\frac{p-1}{p}(z_j^{(p)}-z_j^0)-z_j^{(p)}|=
  \frac{1}{p}|z_j^0-z_j^{(p)}|= \frac{1}{p}\frac{pr_j}{ql_j(z_1^0,z_2^0)}=\frac{r_j}{ql_j(z_1^0,z_2^0)} \label{zetwave}
\end{gather}
From \eqref{zet} we obtain
  $\tilde z^{(p)} \in \mathbb{D}^2 \left[z^0,\frac{(p-1)R}{q(R)\mathbf{L}(z^0)}\right]$ and thus
\begin{gather*}
 S^{*}_{p-1}(z^0,R)\geq 
 \frac{|F^{(k_1^{(p)},k_2^{(p)})}(\tilde z^{(p)})|}{k_1^{(p)}!k_2^{(p)}!l_1^{k_1^{(p)}} (z_1^0,z_2^0)l_2^{k_2^{(p)}} (z_1^0,z_2^0)}
\end{gather*}
 From \eqref{stars} it follows that
\begin{gather}
  0\leq S^{*}_p(z^0,R)-S^{*}_{p-1}(z^0,R) \leq \frac{|F^{(k_1^{(p)},k_2^{(p)})}(z^{(p)})|-|F^{(k_1^{(p)},k_2^{(p)})}(\tilde z^{(p)})|}{k_1^{(p)}!k_2^{(p)}!l_1^{k_1^{(p)}} (z_1^0,z_2^0)l_2^{k_2^{(p)}} (z_1^0,z_2^0)}=\nonumber \\
  = \frac{1}{k_1^{(p)}!k_2^{(p)}!l_1^{k_1^{(p)}} (z_1^0,z_2^0)l_2^{k_2^{(p)}} (z_1^0,z_2^0)}\int_{0}^{1}
  \frac{d}{dt}|F^{(k_1^{(p)},k_2^{(p)})}(\tilde z^{(p)}+t(z^{(p)}-\tilde z^{(p)}))|dt \leq
  \nonumber \\
  \leq \frac{1}{k_1^{(p)}!k_2^{(p)}!l_1^{k_1^{(p)}} (z_1^0,z_2^0)l_2^{k_2^{(p)}} (z_1^0,z_2^0)}\int_{0}^{1}
  (|z_1^{(p)}-\tilde z_1^{(p)} |\cdot |F^{(k_1^{(p)}+1,k_2^{(p)})}(\tilde z^{(p)}+t(z^{(p)}-\tilde z^{(p)}))|+ \nonumber \\
  +|z_2^{(p)}-\tilde z_2^{(p)}|\cdot |F^{(k_1^{(p)},k_2^{(p)}+1)}(\tilde z^{(p)}+t(z^{(p)}-\tilde z^{(p)}))|)dt =\nonumber \\
  = \frac{1}{k_1^{(p)}!k_2^{(p)}!l_1^{k_1^{(p)}} (z_1^0,z_2^0)l_2^{k_2^{(p)}} (z_1^0,z_2^0)}
  (|z_1^{(p)}-\tilde z_1^{(p)}| \cdot |F^{(k_1^{(p)}+1,k_2^{(p)})}(\tilde z^{(p)}+t^*(z^{(p)}-\tilde z^{(p)}))|+ \nonumber \\
  +|z_2^{(p)}-\tilde z_2^{(p)}| \cdot |F^{(k_1^{(p)},k_2^{(p)}+1)}(\tilde z^{(p)}+t^*(z^{(p)}-\tilde z^{(p)}))|),
  \label{big}
\end{gather}
where $0\leq t^*\leq 1, \tilde z^{(p)}+t^*(z^{(p)}-\tilde z^{(p)}) \in \mathbb{D}^2 (z^0,\frac{pR}{q\mathbf{L}(z^0)})$.
For $z\in \mathbb{D}^2 (z^0,\frac{pR}{q\mathbf{L}(z^0)})$ and $j=(j_1,j_2)$, $j_1+j_2\leq N+1$ we have
\begin{gather*}
  \frac{|F^{(j_1,j_2)}(z_1,z_2)|l_1^{j_1} (z_1,z_2)l_2^{j_2} (z_1,z_2)}{j_1!j_2!l_1^{j_1} (z_1^0,z_2^0)l_2^{j_2} (z_1^0,z_2^0)l_1^{j_1} (z_1,z_2)l_2^{j_2} (z_1,z_2)} \leq \\
  \leq (\lambda_{2,1}(R))^{j_1}(\lambda_{2,2}(R))^{j_2}\max\left\{\frac{|F^{(k_1,k_2)}(z_1,z_2)|}{k_1!k_2!l_1^{k_1} (z_1,z_2)l_2^{k_2} (z_1,z_2)}: 0\leq k_1+k_2\leq N \right\} \leq \\
  \leq \prod_{j=1}^{2} (\lambda_{2,j}(R))^{N+1}(\lambda_{1,j}(R))^{-N} \max\left\{\frac{|F^{(k_1,k_2)}(z_1,z_2)|}{k_1!k_2!l_1^{k_1} (z_1^0,z_2^0)l_2^{k_2} (z_1^0,z_2^0)}: 0\leq k_1+k_2\leq N \right\} \leq \\
  \leq \prod_{j=1}^{2} (\lambda_{2,j}(R))^{N+1}(\lambda_{1,j}(R))^{-N} S^{*}_p(z^0,R).
\end{gather*}
From \eqref{big} and \eqref{zetwave} we obtain
\begin{gather*}
  0\leq S^{*}_p(z^0,R)-S^{*}_{p-1}(z^0,R) \leq \\ \leq \prod_{j=1}^{2} (\lambda_{2,j}(R))^{N+1}(\lambda_{1,j}(R))^{-N} S^{*}_p(z^0,R)\sum_{j=1}^{2}(k_j^{(p)}+1)l_j(z_1^0,z_2^0)|z_j^{(p)}-\tilde{z}_j^{(p)}|= \\
  \\ = \prod_{j=1}^{2} (\lambda_{2,j}(R))^{N+1}(\lambda_{1,j}(R))^{-N} \frac{S^{*}_p(z^0,R)}{q(R)}\sum_{j=1}^{2}(k_j^{(p)}+1)r_j \leq \\
  \leq \prod_{j=1}^{2} (\lambda_{2,j}(R))^{N+1}(\lambda_{1,j}(R))^{-N} \frac{S^{*}_p(z^0,R)}{q(R)} (N+1)(r_1+r_2) \leq \frac{1}{2}S^{*}_p(z^0,R).
\end{gather*}
This inequality implies
\begin{gather*}
  S^{*}_p(z^0,R) \leq 2S^{*}_{p-1}(z^0,R),
\end{gather*}
and in view of inequalities \eqref{netss2} and \eqref{stars} we have
\begin{gather*}
  S_p(z^0,R) \leq 2 \prod_{j=1}^{2} (\lambda_{1,j}(R))^{-N}S^{*}_{p-1}(z^0,R) \leq
  2 \prod_{j=1}^{2} (\lambda_{1,j}(R))^{-N}(\lambda_{2,j}(R))^{N}S_{p-1}(z^0,R)
\end{gather*}
Therefore,
\begin{gather}
  \max\left\{\frac{|F^{(k_1,k_2)}(z_1,z_2)|}{k_1!k_2!l_1^{k_1} (z_1,z_2)l_2^{k_2} (z_1,z_2)}: 0\le k_1+k_2\leq N, z\in \mathbb{D}^2 \left[z^0,\frac{pR}{q\mathbf{L}(z^0)}\right] \right\}=\nonumber \\
  = S_q(z^0,R) \leq 2 \prod_{j=1}^{2} (\lambda_{1,j}(R))^{-N}(\lambda_{2,j}(R))^{N}S_{q-1}(z^0,R) \leq \ldots \leq \nonumber \\ \leq (2 \prod_{j=1}^{2} (\lambda_{1,j}(R))^{-N}(\lambda_{2,j}(R))^{N})^q S_{0}(z^0,R)=\nonumber \\
  = (2 \prod_{j=1}^{2} (\lambda_{1,j}(R))^{-N}(\lambda_{2,j}(R))^{N})^q \max\left\{\frac{|F^{(k_1,k_2)}(z^0_1,z^0_2)|}{k_1!k_2!l_1^{k_1} (z_1^0,z_2^0)l_2^{k_2} (z_1^0,z_2^0)}: 0\le k_1+k_2\leq N \right\}.
  \label{lastth1}
\end{gather}

From \eqref{lastth1} we obtain inequality \eqref{net1}  with $p_0=(2 \prod_{j=1}^{2} (\lambda_{1,j}(R))^{-N}(\lambda_{2,j}(R))^{N})^q $ and some $k^0=k_1^0+k_2^0\leq N$.
		The necessity of condition \eqref{net1} is proved.
  
Now we prove the sufficiency. Suppose that for every $R\in\mathcal{B}^2$ $\exists n_0 \in \mathbb{Z}_+, p_0>1 $ such that $\forall z_0 \in \mathbb{D}^2 $ and some
$ k_0 \in \mathbb{Z}_+^2, k_1^0+k_2^0\leq n_0,$ the inequality \eqref{net1} holds.

 We write Cauchy's formula as following  
$\forall z^0\in \mathbb{D}^2$ $\forall k\in \mathbb{Z}_+^2$ $\forall s \in \mathbb{Z}_+^2$
	$$
	\frac{F^{(k_1+s_1,k_2+s_2)}(z_1^0,z_2^0)}{s_1!s_2!}=\frac{1}{(2\pi i)^2} \int_{\mathbb{T}^2\left(z^0,\frac{R}{\mathbf{L}(z^0)}\right)} \frac{F^{(k_1,k_2)}(z_1,z_2)}{(z_1-z_1^0)^{s_1+1}(z_2-z_2^0)^{s_2+1}} dz_1dz_2.
	$$
 Therefore, applying \eqref{net1}, we have 
\begin{gather*}
  \frac{|F^{(k_1+s_1,k_2+s_2)}(z_1^0,z_2^0)|}{s_1!s_2!}\leq \frac{1}{(2\pi)^2} \int_{\mathbb{T}^2\left(z^0,\frac{R}{\mathbf{L}(z^0)}\right)} \frac{|F^{(k_1,k_2)}(z_1,z_2)|}{|z_1-z_1^0|^{s_1+1}|z_2-z_2^0|^{s_2+1}} |dz_1||dz_2|
  \leq \\
  \\ \leq \frac{1}{(2\pi)^2} \int_{\mathbb{T}^2\left(z^0,\frac{R}{\mathbf{L}(z^0)}\right)} |F^{(k_1,k_2)}(z_1,z_2)| \frac{{l_1^{s_1+1}(z^0)}{l_2^{s_2+1}(z^0)}}{{r_1}^{s_1+1}{r_2}^{s_2+1}} |dz_1||dz_2| \leq \\
  \leq
 \int_{\mathbb{T}^2\left(z^0,\frac{R}{\mathbf{L}(z^0)}\right)} |F^{(k_1^0,k_2^0)}(z_1^0,z_2^0)| \frac{k_1!k_2!p_0(\lambda_{2,1}(R)\lambda_{2,2}(R))^{n_0}{l_1^{s_1+k_1+1}(z^0)}
    {l_2^{s_2+k_2+1}(z^0)}}
    {(2\pi)^2k_1^0!k_2^0!{r_1}^{s_1+1}{r_2}^{s_2+1}{l_1^{k_1^0}(z^0)}{l_2^{k_2^0}(z^0)}} |dz_1||dz_2|=
 \\
 =|F^{(k_1^0,k_2^0)}(z_1^0,z_2^0)| \frac{k_1!k_2!p_0(\lambda_{2,1}(R)\lambda_{2,2}(R))^{n_0}{l_1^{s_1+k_1}(z^0)}{l_2^{s_2+k_2}(z^0)}}
    {k_1^0!k_2^0!{r_1}^{s_1}{r_2}^{s_2}{l_1^{k_1^0}(z^0)}{l_2^{k_2^0}(z^0)}}.
\end{gather*}
This implies
\begin{gather}
 \! \frac{|F^{(k_1+s_1,k_2+s_2)}(z_1^0,z_2^0)|}{(k_1+s_1)!(k_2+s_2)!{l_1^{s_1+k_1}(z^0)}{l_2^{s_2+k_2}(z^0)}} 
  \leq \frac{(\lambda_{2,1}(R)\lambda_{2,2}(R))^{n_0}p_0k_1!k_2!s_1!s_2!}{(k_1+s_1)!(k_2+s_2)!r_1^{s_1}r_2^{s_2}}
  \frac{|F^{(k_1^0,k_2^0)}(z_1^0,z_2^0)|}{k_1^0!k_2^0!{l_1^{k_1^0}(z^0)}{l_2^{k_2^0}(z^0)}}. \label{dopner}
  \end{gather}
  Obviously, that 
  $$\frac{k_1!k_2!s_1!s_2!}{(k_1+s_1)!(k_2+s_2)!}=\frac{s_1!}{(k_1+1)\cdot\ldots \cdot (k_1+s_1)} \frac{s_2!}{(k_2+1)\cdot\ldots \cdot (k_2+s_2)}\le1.$$ 
  We choose $r_j\in(1,\beta],$ $j\in\{1,2\}.$ 
  Hence, 
  $$
  \frac{(\lambda_{2,1}(R)\lambda_{2,2}(R))^{n_0}p_0}{r_1^{s_1}r_2^{s_2}} \rightarrow 0 \text{ as } s_1+s_2\to +\infty.
  $$
  
Thus, there exists $s_0$ such that as $s_1+s_2\geq s_0$ the inequality holds:
   $$
  \frac{(\lambda_{2,1}(R)\lambda_{2,2}(R))^{n_0}p_0k_1!k_2!s_1!s_2!}{(k_1+s_1)!(k_2+s_2)!r_1^{s_1}r_2^{s_2}}\leq 1.
  $$
  Inequality \eqref{dopner} yields that
  \begin{gather*}
  \frac{|F^{(k_1+s_1,k_2+s_2)}(z_1^0,z_2^0)|}{(k_1+s_1)!(k_2+s_2)!{l_1}^{k_1+s_1}(z^0){l_2}^{k_2+s_2}(z^0)} \leq
  \frac{|F^{(k_1^0,k_2^0)}(z_1^0,z_2^0)|}{k_1^0!k_2^0!{l_1}^{k_1^0}(z^0){l_2}^{k_2^0}(z^0)}.
  \end{gather*}
This means that for every $ j\in \mathbb{Z}_+^2$ 
  \begin{gather*}
   \frac{|F^{(j_1,j_2)}(z_1^0,z_2^0)|}{j_1!j_2! l_1^{j_1}(z^0) l_2^{j_2}(z^0)} \leq
\max\left\{  \frac{|F^{(k_1,k_2)}(z_1^0,z_2^0)|}{k_1!k_2!{l_1^{k_1}(z^0)}{l_2^{k_2}(z^0)}}:  k_1+k_2\le s_0+n_0\right\}
  \end{gather*}
  where $s_0$ and $n_0$ are independent of $z_0$. Therefore, the function $F$ has bounded $\mathbf{L}$-index in joint variables $N\le s_0+n_0.$
\end{proof}

\begin{theorem}
\label{cor1}
Let $\mathbf{L} \in Q^2(\mathbb{D}^2).$ In order that an analytic in $\mathbb{D}^2$ function $F$ be of bounded $\mathbf{L}$-index in joint variables it is necessary that for every $R\in\mathcal{B}^2$ $\exists n_0 \in \mathbb{Z}_+$ $\exists p\geq 1$ $\forall z^0 \in \mathbb{D}^2$ $\exists k^0 \in \mathbb{Z}_+^2 $, $k_1^0+k_2^0 \leq n_0,$ and  
\begin{gather}
 \max\left \{ |F^{(k_1^0,k_2^0)}(z_1,z_2)|:z \in \mathbb{D}^2\left[z^0, \frac{R}{\mathbf{L}(z^0)}\right] \right\} \leq p|F^{(k_1^0,k_2^0)}(z_1^0,z_2^0)|
 \label{conc1}
\end{gather}
 and it is sufficient that for every $R\in\mathcal{B}^2$  $\exists n_0 \in \mathbb{Z}_+$ $\exists p\leq 1$ $\forall z^0 \in \mathbb{D}^2$  $\exists k_1^0 \leq n_0$ $\exists k_2^0 \leq n_0$ and 
\begin{gather}
  \max \left\{ |F^{(k_1^0,0)}(z_1,z_2)|:z \in \mathbb{D}^2\left[z^0, \frac{R}{\mathbf{L}(z^0)}\right] \right\} \leq p|F^{(k_1^0,0)}(z_1^0,z_2^0)|
 \label{conc3}\\
 \max \left\{ |F^{(0,k_2^0)}(z_1,z_2)|:z \in \mathbb{D}^2\left[z^0, \frac{R}{\mathbf{L}(z^0)}\right] \right\} \leq p|F^{(0,k_2^0)}(z_1^0,z_2^0)|.
 \label{conc4}
\end{gather}
\end{theorem}
\begin{proof}
  Proof of Theorem \ref{petr1} implies that the inequality \eqref{net1}
  is true for some $k^0.$
  Therefore, we have
  \begin{gather*}
    \frac{p_0}{k_1^0!k_2^0!}  \frac{|F^{(k_1^0,k_2^0)}(z_1^0,z_2^0)|}{l_1^{k_1^0} (z_1^0,z_2^0)l_2^{k_2^0} (z_1^0,z_2^0)}\geq \\
  \geq
  \max\left\{ \frac{1}{k_1^0!k_2^0!}  \frac{|F^{(k_1^0,k_2^0)}(z_1,z_2)|}{l_1^{k_1^0} (z_1,z_2)l_2^{k_2^0} (z_1,z_2)}: z\in \mathbb{D}^2 \left[z^0,\frac{R}{\mathbf{L}(z^0)}\right] \right\} =
  \\
  =
  \max\left\{ \frac{|F^{(k_1^0,k_2^0)}(z_1,z_2)|}{k_1^0!k_2^0!}  \frac{l_1^{k_1^0} (z_1^0,z_2^0)l_2^{k_2^0} (z_1^0,z_2^0)}{l_1^{k_1^0} (z_1^0,z_2^0)l_2^{k_2^0} (z_1^0,z_2^0)l_1^{k_1^0} (z_1,z_2)l_2^{k_2^0} (z_1,z_2)}: z\in \mathbb{D}^2 \left[z^0,\frac{R}{\mathbf{L}(z^0)}\right] \right\} \geq
  \\
  \geq
  \max\left\{ \frac{|F^{(k_1^0,k_2^0)}(z_1,z_2)|}{k_1^0!k_2^0!}  \frac{{(\lambda_{2,1}(R)\lambda_{2,2}(R))}^{-n_0}}{l_1^{k_1^0} (z_1^0,z_2^0)l_2^{k_2^0} (z_1^0,z_2^0)}: z\in \mathbb{D}^2 \left[z^0,\frac{R}{\mathbf{L}(z^0)}\right] \right\}.
  \end{gather*}
  This inequality implies
  \begin{gather}
    \frac{p_0{(\lambda_{2,1}(R)\lambda_{2,2}(R))}^{n_0}}{k_1^0!k_2^0!}  \frac{|F^{(k_1^0,k_2^0)}(z_1^0,z_2^0)|}{l_1^{k_1^0} (z_1^0,z_2^0)l_2^{k_2^0} (z_1^0,z_2^0)}\geq \nonumber\\
  \geq
  \max\left\{ \frac{1}{k_1^0!k_2^0!}  \frac{|F^{(k_1^0,k_2^0)}(z_1,z_2)|}{l_1^{k_1^0} (z_1^0,z_2^0)l_2^{k_2^0} (z_1^0,z_2^0)}: z\in \mathbb{D}^2 \left[z^0,\frac{R}{\mathbf{L}(z^0)}\right] \right\}.
  \label{conc2}
    \end{gather}
    From \eqref{conc2} we obtain inequality \eqref{conc1} with
    $p=p_0{(\lambda_{2,1}(R)\lambda_{2,2}(R))}^{n_0}$.
  The necessity of condition \eqref{conc1} is proved.
  
Now we prove the sufficiency of \eqref{conc3} and \eqref{conc4}. Suppose that for every $R\in\mathcal{B}^2$ $\exists n_0 \in \mathbb{Z}_+, p>1 $ such that $\forall z_0 \in \mathbb{D}^2 $ and some
$ k_1^0\leq n_0, \ \ k_2^0\leq n_0$ the inequalities \eqref{conc3} and \eqref{conc4} hold.

We write Cauchy's formula as following 
$\forall z^0\in \mathbb{D}^2$ $\forall k^0_1\in \mathbb{Z}$ $\forall s \in \mathbb{Z}_+^2 $
	$$
	\frac{F^{(k_1^0+s_1,s_2)}(z_1^0,z_2^0)}{s_1!s_2!}=\frac{1}{(2\pi i)^2} \int_{T^2\left(z^0,\frac{R}{\mathbf{L}(z^0)}\right)} \frac{F^{(k_1^0,0)}(z_1,z_2)}{(z_1-z_1^0)^{s_1+1}(z_2-z_2^0)^{s_2+1}} dz_1dz_2.
	$$
 This yields 
\begin{gather*}
  \frac{|F^{(k_1^0+s_1,s_2)}(z_1^0,z_2^0)|}{s_1!s_2!}\leq \frac{1}{(2\pi)^2} \int_{\mathbb{T}^2\left(z^0,\frac{R}{\mathbf{L}(z^0)}\right)} \frac{|F^{(k_1,k_2)}(z_1,z_2)|}{|z_1-z_1^0|^{s_1+1}|z_2-z_2^0|^{s_2+1}} |dz_1||dz_2|
  \leq \\
  \\ \leq \frac{1}{(2\pi)^2} \int_{\mathbb{T}^2\left(z^0,\frac{R}{\mathbf{L}(z^0)}\right)} \max\{|F^{(k_1,0)}(z_1,z_2)|:z\in \mathbb{D}^2\} \frac{{l_1^{s_1+1}(z^0)}{l_2^{s_2+1}(z^0)}}{{r_1}^{s_1+1}{r_2}^{s_2+1}} |dz_1||dz_2| = \\
  =
  \max\{|F^{(k_1,0)}(z_1,z_2)|:z\in \mathbb{D}^2\} \frac{{l_1^{s_1}(z^0)}{l_2^{s_2}(z^0)}}{{r_1}^{s_1}{r_2}^{s_2}}.
\end{gather*}
Now we put $r_1=r_2=\beta$ and use \eqref{conc3} 
\begin{gather}
\frac{|F^{(k_1^0+s_1,s_2)}(z_1^0,z_2^0)|}{s_1!s_2!} \leq
 \frac{{l_1^{s_1}(z^0)}{l_2^{s_2}(z^0)}}{{\beta}^{s_1+s_2}}
 \max\{|F^{(k_1,0)}(z_1,z_2)|:z\in \mathbb{D}^2\}\leq \nonumber \\
  \leq
\frac{p{l_1^{s_1}(z^0)}{l_2^{s_2}(z^0)}}{{\beta}^{s_1+s_2}}
|F^{(k_1^0,0)}(z_1^0,z_2^0)| \label{eqa1}
\end{gather}
We choose $s_1+s_2\geq s_0$, where $\frac{p}{\beta^{s_0}} \leq 1$. 
Therefore \eqref{eqa1} implies 
\begin{gather*}
  \frac{|F^{(k_1^0+s_1,s_2)}(z_1^0,z_2^0)|}
  {{l_1^{k_1^0+s_1}(z_1^0,z_2^0}{l_2^{s_2}(z_1^0,z_2^0)}(k_1^0+s_1)!s_2!} \leq
  \frac{p}{{\beta}^{s_1+s_2}} \frac{s_1!k_1^0!}{(s_1+k_1^0)!}
  \frac{|F^{(k_1^0,0)}(z_1^0,z_2^0)|}
  {{l_1^{k_1^0}(z_1^0,z_2^0)}k_1^0!} \leq
  \frac{|F^{(k_1^0,0)}(z_1^0,z_2^0)|}{{l_1^{k_1^0}(z_1^0,z_2^0)}k_1^0!}.
\end{gather*}
Similarly
\begin{gather*}
 \frac{|F^{(s_1,k_2^0+s_2)}(z_1^0,z_2^0)|}
  {{l_1^{s_1}(z_1^0,z_2^0)}{l_2^{k_2^0+s_2}(z_1^0,z_2^0)}(k_2^0+s_2)!s_1!} \leq
    \frac{|F^{(0,k_2^0)}(z_1^0,z_2^0)|}{{l_2^{k_2^0}(z_1^0,z_2^0)}k_2^0!}.
\end{gather*}
Consequently, $N(F,\mathbf{L},\mathbb{D}^2)\leq n_0+s_0,$ as $k_1^0\leq n_0, k_2^0\leq n_0.$
\end{proof}

\begin{remark}
Note that necessity of Theorem \ref{cor1} was established by Bordulyak M. T. and Sheremeta M. M. \cite{bagzmin}, Bandura A. I., Bordulyak M. T. and Skaskiv O. B. \cite{sufjointdir} for entire functions of several variables.
But they did not obtain sufficiency in this case, although inequality \eqref{conc1} is necessary and sufficient condition of boundedness of $l$-index for functions of one variable \cite{sher,kusher,sherkuz}. Our restrictions \eqref{conc3}-\eqref{conc4} are corresponding multidimensional sufficient conditions. Moreover, assumptions \eqref{conc3} and \eqref{conc4} provide boundedness of $l_1$- and $l_2$-index in the directions $(1,0)$ and $(0,1)$ accordingly (see definition and properties for entire functions in \cite{BandSk,monograph}).  As a matter of fact, we implicitly deduce 
property similar to Theorem 6 in \cite{sufjointdir}. The theorem state that if an entire in $\mathbb{C}^n$ function $F$ has bounded $l_j$-index in a direction $e_j$ for every $j\in\{1, . . . , n\},$ then $F$ is of bounded $\mathbf{L}$-index in joint variables, 
where $\mathbf{L}=(l_1,\ldots,l_n),$  $\mathbf{e}_j=(0,\ldots,0, \underbrace{1}_{j-\mbox{th place}}, 0,\ldots,0).$
  \end{remark}

Denote $\tilde{\mathbf{L}}(z_1,z_2)=(\tilde{l}_1(z_1,z_2),\tilde{l}_2(z_1,z_2))$. $\mathbf{L}\asymp \tilde{\mathbf{L}}$ means that $\exists \varTheta_j=(\theta_{1,j},\theta_{2,j})\in \mathbb{R}_+^2,$ $j\in\{1,2\}$ such that $\forall (z_1,z_2) \in \mathbb{D}^2$
$$
\theta_{1,j}\tilde{l}_j(z_1,z_2) \leq l_j(z_1,z_2)\leq \theta_{2,j}\tilde{l}_j(z_1,z_2).
$$

 \begin{theorem} \label{petr2}
 Let $\mathbf{L} \in Q^2(\mathbb{D}^2)$ and $\mathbf{L}\asymp \tilde{\mathbf{L}}$. An analytic function $F$ in $\mathbb{D}^2$  has bounded  $\tilde{\mathbf{L}}$-index in joint variables if and only if it has bounded $\mathbf{L}$-index.
\end{theorem}
\begin{proof}
  It is easy to prove that if $\mathbf{L} \in Q^2(\mathbb{D}^2)$ and $\mathbf{L}\asymp \tilde{\mathbf{L}}$ then $\tilde{\mathbf{L}} \in Q^2(\mathbb{D}^2).$

  Let $N(F,\tilde{\mathbf{L}},\mathbb{D}^2)=\tilde{n}_0<+\infty$. Then by Theorem \ref{petr1}1 for every $\tilde{R}=(\tilde{r}_1,\tilde{r}_2)\in\mathcal{B}^2$ there exists $\tilde{p}\ge 1$ such that for each $z^0 \in \mathbb{D}^2$ and some $k^0$, $k_1^0+k_2^0\leq \tilde{n}_0,$ the inequality \eqref{net1} holds with $\tilde{\mathbf{L}}$ and $\tilde{R}$ instead of $\mathbf{L}$ and $R$.
  Hence
  \begin{gather*}
    \frac{\tilde{p}}{k_1^0!k_2^0!}  \frac{|F^{(k_1^0,k_2^0)}(z_1^0,z_2^0)|}{l_1^{k_1^0} (z_1^0,z_2^0)l_2^{k_2^0} (z_1^0,z_2^0)}=
     \frac{\tilde{p}}{k_1^0!k_2^0!}  \frac{\theta_{2,1}^{k_1^0}\theta_{2,2}^{k_2^0} |F^{(k_1^0,k_2^0)}(z_1^0,z_2^0)|}
     {\theta_{2,1}^{k_1^0}\theta_{2,2}^{k_2^0} l_1^{k_1^0} (z_1^0,z_2^0)l_2^{k_2^0} (z_1^0,z_2^0)} \geq
     \\ \geq
     \frac{\tilde{p}}{k_1^0!k_2^0!}  \frac{|F^{(k_1^0,k_2^0)}(z_1^0,z_2^0)|}
     {\theta_{2,1}^{k_1^0}\theta_{2,2}^{k_2^0} \tilde{l}_1^{k_1^0} (z_1^0,z_2^0)\tilde{l}_2^{k_2^0} (z_1^0,z_2^0)} \geq
     \\ \geq
     \frac{1}{\theta_{2,1}^{k_1^0}\theta_{2,2}^{k_2^0} }\max \left\{ \frac{|F^{(k_1,k_2)}(z_1,z_2)|}
     {k_1!k_2!\tilde{l}_1^{k_1} (z_1,z_2)\tilde{l}_2^{k_2} (z_1,z_2)}: k_1+k_2 \leq \tilde{n}_0, z \in \mathbb{D}^2\left[z^0, \frac{\tilde{R}}{\tilde{\mathbf{L}}(z)}\right] \right\} \geq
     \\ \geq
      \frac{1}{\theta_{2,1}^{k_1^0}\theta_{2,2}^{k_2^0} }\max \left\{ \frac{\theta_{1,1}^{k_1}\theta_{1,2}^{k_2} |F^{(k_1,k_2)}(z_1,z_2)|}
     {k_1!k_2!l_1^{k_1} (z_1,z_2)l_2^{k_2} (z_1,z_2)}: k_1+k_2 \leq \tilde{n}_0, z \in \mathbb{D}^2\left[z^0, \frac{\tilde{R}}{\tilde{\mathbf{L}}(z)}\right] \right\} \geq \\
     \geq
     \frac{\min \limits_{0\! \leq\! k_1+k_2\!\leq\! n_0
     } \{\theta_{1,1}^{k_1}\theta_{1,2}^{k_2} \}}{\theta_{2,1}^{k_1^0}\theta_{2,2}^{k_2^0} }
     \max \left\{ \frac{|F^{(k_1,k_2)}(z_1,z_2)|}
     {k_1!k_2!l_1^{k_1} (z_1,z_2)l_2^{k_2} (z_1,z_2)}: k_1\!+\!k_2\! \leq\! \tilde{n}_0, z \!\in\! \mathbb{D}^2\!\left[z^0, \frac{\tilde{R}}{\tilde{\mathbf{L}}(z)}\right] \right\}.
  \end{gather*}
  In view of Theorem \ref{petr1} we obtain that function $F$ has bounded $\mathbf{L}$-index.
\end{proof}
\begin{theorem}
 Let $\mathbf{L} \in Q^2(\mathbb{D}^2)$. An analytic function $F$ in $\mathbb{D}^2$  has bounded  $\tilde{\mathbf{L}}$-index in joint variables if and only if there exist $R\in\mathcal{B}^2,$ $n_0 \in \mathbb{Z}_+,$ $p_0>1$ such that for each $z^0 \in \mathbb{D}^2(z^0,R)$ and for some $k^0 \in \mathbb{Z}_+^2, k_1^0+k_2^0 \leq n_0$ the inequality \eqref{net1} holds.
\end{theorem}
\begin{proof}
  The sufficiency of this theorem follows from the sufficiency of Theorem \ref{petr1}.
  We prove the necessity.
  The proof of Theorem \ref{petr1} with $R=(\beta,\beta)$ implies that $N(F,L,\mathbb{D}^2)<+\infty.$
Let  $\mathbf{L}^{*}=(\frac{\beta l_1(z_1,z_2)}{r_1},\frac{\beta l_2(z_1,z_2)}{r_2}),$ $R^0=(\beta,\beta).$
  In general case from validity of \eqref{net1} for $F$ and $\mathbf{L}$ with $R=(r_1,r_2),$ $r_j<\beta,$ $j\in\{1,2\}$  we obtain 
  \begin{gather*}
    \max \left\{ \frac{|F^{(k_1,k_2)}(z_1,z_2)|}
     {k_1!k_2!(\beta l_1 (z_1,z_2)/r_1)^{k_1}(\beta l_2 (z_1,z_2)/r_2)^{k_2}}: k_1+k_2 \leq n_0, z \in \mathbb{D}^2\left[z^0, \frac{R_0}{\mathbf{L}^*(z^0)}\right] \right\} \leq \\
     \leq
      \max \left\{ \frac{|F^{(k_1,k_2)}(z_1,z_2)|}
     {k_1!k_2!l_1^{k_1} (z_1,z_2)l_2^{k_2} (z_1,z_2)}: k_1+k_2 \leq n_0,  z\! \in\! \mathbb{D}^2\left[z^0, \frac{R}{\mathbf{L}(z^0)}\right] \right\} \leq
     \\
     \leq 
     \frac{p_0}{k_1^0!k_2^0!} \frac{|F^{(k_1^0,k_2^0)}(z_1^0,z_2^0)|}{l_1^{k_1^0} (z_1^0,z_2^0)l_2^{k_2^0} (z_1^0,z_2^0)}
     =
     \frac{\beta^{k_1^0+k_2^0}p_0}{r_1^{k_1^0}r_2^{k_2^0}k_1^0!k_2^0!} \frac{|F^{(k_1^0,k_2^0)}(z_1,z_2)|}{(\beta l_1(z_1^0,z_2^0)/r_1)^{k_1^0}(\beta l_2 (z_1^0,z_2^0)/r_2)^{k_2^0}}<\\
     < \frac{p_0\beta^{2n_0}}{(r_1r_2)^{n_0}} \frac{|F^{(k_1^0,k_2^0)}(z_1,z_2)|}{k_1^0!k_2^0!(\beta l_1(z_1^0,z_2^0)/r_1)^{k_1^0}(\beta l_2 (z_1^0,z_2^0)/r_2)^{k_2^0}}.
       \end{gather*}
       i. e. \eqref{net1} holds for $F,$  $\mathbf{L}^*$ and $R=(\beta,\beta).$
       Now as above for $R=(\beta,\beta)$ we apply Theorem \ref{petr1} for function $F(z_1,z_2)$ and $\mathbf{L}^{*}(z_1,z_2)=(\frac{l_1(z_1,z_2)}{r_1},\frac{l_2(z_1,z_2)}{r_2})$. This implies that $F$ is of bounded $\mathbf{L}^{*}$-index in joint variables. Therefore, by Theorem \ref{petr2} the function $F$ is of bounded $\mathbf{L}$-index in joint variables.
\end{proof}

  \noindent{\bf 4. Estimate of maximum modulus on a bidisc.}

For an entire function $F(z)$ we put
$$M(R,z^0,F)=\max\{|F(z)|\colon  z\in \mathbb{T}^2(z^0,R)\}.$$ Then $M(R,z^0,F)=\max\{|F(z)|  \colon z\in \mathbb{D}^2[z^0,R]\},$ because the maximum modulus for an entire function in a closed polydisc is  attained on its skeleton.

\begin{theorem} \label{petr4}
 Let $\mathbf{L} \in Q^2(\mathbb{D}^2)$. An analytic function $F$ in $\mathbb{D}^2$  has bounded  $\mathbf{L}$-index in joint variables if and only if for any $R',R'' \in \mathbb{R}_+^2$, $\mathbf{0}<R'<R''\leq (\beta,\beta)$ there exists $p_1=p_1(R',R'')\geq 1$ such that for each $z^0 \in \mathbb{D}^2$
\begin{gather}\label{th4main}
M\left(\frac{R''}{\mathbf{L}(z^0)},z^0,F\right)\leq
pM\left(\frac{R'}{\mathbf{L}(z^0)},z^0,F\right).
\end{gather} 
\end{theorem}
\begin{proof}
 Let $N(F,L,\mathbb{D}^2)=N<+\infty.$ Suppose that inequality \eqref{th4main} does not hold i.e. there exist
	$R',$ $R'',$ $\mathbf{0}<R'<R'',$ such that for each $p_*\geq 1$ and for some $z^0=z^0(p_*)$
 \begin{gather}\label{th4f}
M\left(\frac{R''}{\mathbf{L}(z^0)},z^0,F\right)>
p_*M\left(\frac{R'}{\mathbf{L}(z^0)},z^0,F\right).
\end{gather}
By Theorem \ref{cor1} there exists a number $p_0=p_0(R'')\geq 1$ such that for every $z^0\in\mathbb{D}^2$ and for some $k^0\in\mathbb{Z}^2_{+},$ $k_1^0+k_2^0 \leq N,$  one has
	\begin{equation}
	\label{th4s}
	M\left(\frac{R''}{\mathbf{L}(z^0)},z^0,F^{(k_1^0,k_2^0)}\right) \leq p_0 |F^{(k_1^0,k_2^0)}(z^0)|.
	\end{equation}
	We put
\begin{gather*}
b_1=p_0N!
\left(\frac{r''_1r''_2}{r'_1r'_2}\right)^{N}
\lambda_{2,1}^{N}(R'')\lambda_{2,2}^{N}(R'')
\sum_{j=1}^{N} \frac{(N-j)!}{(r_1^{''})^{j}}
 \\
  b_2=p_0
\lambda_{2,2}^{N}(R'')
\sum_{j=1}^{N} \frac{(N-j)!}{(r_2^{''})^{j}}\max\left\{ \frac{1}{(r_1^{''})^N},1 \right\}
  \\
  p_*=p_0(N!)^2\left(\frac{r_1^{''}r_2^{''}}{r_1^{'}r_2^{'}}\right)^N+b_1+b_2+1.
\end{gather*}

	Let $z^0=z^0(p_*)$ be a point for which  inequality \eqref{th4f} holds and $k^0$ is such for which \eqref{th4s} holds. We choose $z^*$ and $z_{(j_1,j_2)}^*$ such that
\begin{gather*}
 M\left(\frac{R'}{\mathbf{L}(z^0)},z^0,F\right)=|F(z^*)|, \;
 M\left(\frac{R''}{\mathbf{L}(z^0)},z^0,F^{(j_1,j_2)}\right)=|F^{(j_1,j_2)}(z_{(j_1,j_2)}^*)|
\end{gather*}
for every $j=(j_1,j_2)\in\mathbb{Z}^2_{+},$ $j_1+j_2\leq N.$
	We apply Cauchy's inequality
\begin{gather}\label{th4thi}
  |F^{(j_1,j_2)}(z^0)|\leq j_1!j_2!\left( \frac{l_1(z_1^0,z_2^0)}{r_1^{'}}  \right)^{j_1} \left( \frac{l_2(z_1^0,z_2^0)}{r_2^{'}} \right)^{j_2}|F(z^*)|
\end{gather}
for estimate the difference
\begin{gather}
  |F^{(j_1,j_2)}(z_{j,1}^{*},z_{j,2}^{*})-F^{(j_1,j_2)}(z_{1}^0,z_{j,2}^{*})|=\left |\int_{z_1^0}^{z_{j,1}^{*}} F^{(j_1+1,j_2)}(\zeta,z_{j,2}^{*})d\zeta\right| \leq \nonumber \\
  \leq
\int_{z_1^0}^{z_{j,1}^{*}}\max \left\{ |F^{(j_1+1,j_2)}(\zeta,z_{j,2}^{*})|:|\zeta-z_1^0|=
  \frac{r_1^{''}}{l_1(z_0)} \right\} |d \zeta|=
 |F^{(j_1+1,j_2)}(z^*_{(j_1+1,j_2)})|\frac{r_1^{''}}{l_1(z^0)}.
  \label{mes}
\end{gather}
Since $(z_1^0,z_{j,2}^{*}) \in \mathbb{D}^2 \left[ z^0, \frac{R''}{\mathbf{L}(z^0)} \right]  $
and
for all $k=1,2$ we have that
$|z_{j,k}^{*}-z_k^0|=\frac{r_k^{''}}{l_k(z^0)} $
and
$l_k(z_1^0,z_{j,2}^{*}) \leq \lambda_{2,k}(R{''})l_k(z^0)$. 
Putting $j=k^0$ in \eqref{th4thi}, by Theorem \ref{petr1} we obtain that
\begin{gather}
|F^{(j_1,j_2)}(z_1^0,z_{j,2}^{*})|\leq
j_1!j_2! p_0|F^{(k^0)}(z^0)|
   \frac{l_1^{j_1}(z_1^0,z_{j,2}^{*})l_2^{j_2}(z_1^0,z_{j,2}^{*})}{k^0_1!k_2^0!l^{k_1^0}_1(z_1^0,z_2^0)l^{k_2^0}_2(z_1^0,z_2^0)}
   \leq \nonumber \\
   \leq
   \frac{j_1!j_2!\lambda_{2,1}^{j_1}(R'')\lambda_{2,2}^{j_2}(R'')
   l_1^{j_1}(z^0)l_2^{j_2}(z^0)}{k^0_1!k_2^0!l^{k_1^0}_1(z_1^0,z_2^0)l^{k_2^0}_2(z_1^0,z_2^0)}p_0k_1^0!k_2^0!
  \left( \frac{l_1(z_1^0,z_2^0)}{r_1^{'}}  \right)^{k_1^0} \left( \frac{l_2(z_1^0,z_2^0)}{r_2^{'}} \right)^{k^0_2}|F(z^{*})| \leq \nonumber \\
   \leq
   j_1!j_2!\lambda_{2,1}^{j_1}(R'')\lambda_{2,2}^{j_2}(R'')p_0
    \frac{l_1^{j_1}(z_1^0,z_2^0)l_2^{j_2}(z_1^0,z_2^0)}{(r_1')^{k_1^0}(r_2')^{k_2^0}} |F(z^*)|.
   \label{th4ni}
\end{gather}
From inequalities \eqref{mes} and \eqref{th4ni}  it follows that
\begin{gather*}
 |F^{(j_1+1,j_2)}(z^*_{(j_1+1,j_2)})|\geq \frac{l_1(z^0)}{r_1^{''}}
(|F^{(j_1,j_2)}(z_{j,1}^{*},z_{j,2}^{*})|-|F^{(j_1,j_2)}(z_1^0,z_{j,2}^{*})|)\geq \nonumber \\
\geq
\frac{l_1(z^0)}{r_1^{''}}(|F^{(j_1,j_2)}(z_{j,1}^{*},z_{j,2}^{*})|-
j_1!j_2!\lambda_{2,1}^{j_1}(R'')\lambda_{2,2}^{j_2}(R'')p_0
   \frac{l_1^{j_1}(z_1^0,z_2^0)l_2^{j_2}(z_1^0,z_2^0)}{(r_1')^{k_1^0}(r_2')^{k_2^0}}  |F(z^*)|)= \nonumber \\
   =\frac{l_1(z^0)}{r_1^{''}}|F^{(j_1,j_2)}(z_{j,1}^{*},z_{j,2}^{*})|-
j_1!j_2!\lambda_{2,1}^{j_1}(R'')\lambda_{2,2}^{j_2}(R'')p_0l_1(z^0)
    \frac{l_1^{j_1}(z_1^0,z_2^0)l_2^{j_2}(z_1^0,z_2^0)}{r_1''(r_1')^{k_1^0}(r_2')^{k_2^0}} |F(z^*)|.
\end{gather*}
We choose $j=(j_1,j_2)=(k_1^0,k_2^0)$ and deduce 
\begin{gather}
|F^{(k_1^0,k_2^0)}(z_{k^0}^{*})|\geq \frac{l_1(z^0)}{r_1^{''}}
|F^{(k_1^0-1,k_2^0)}(z_{(k_1^0-1,k_2^0)}^{*})|-
\frac{p_0(k_1^0-1)!k_2^0!l_1^{k_1^0}(z_1^0,z_2^0)l_2^{k_2^0}(z_1^0,z_2^0)}{r_1^{''}(r'_1)^{k_1^0}(r'_2)^{k_2^0}} \times \nonumber \\ 
\!\times\! \lambda_{2,1}^{j_1}(R'')\lambda_{2,2}^{j_2}(R'')|F(z^{*})| 
\!\geq\!
\frac{l_1^2(z^0)}{(r_1^{''})^2}|F^{(k_1^0-2,k_2^0)}(z_{(k_1^0-2,k_2^0)}^{*})|\!-\!
\frac{p_0(k_1^0-2)!k_2^0!l_1^{k_1^0}(z_1^0,z_2^0)l_2^{k_2^0}(z_1^0,z_2^0)}{(r_1^{''})^2(r'_1)^{k_1^0}(r'_2)^{k_2^0}} \times \nonumber \\ \times 
\lambda_{2,1}^{j_1}(R'')\lambda_{2,2}^{j_2}(R'')|F(z^{*})|- 
\frac{p_0(k_1^0-1)!k_2^0!l_1^{k_1^0}(z_1^0,z_2^0)l_2^{k_2^0}(z_1^0,z_2^0)}{r_1^{''}(r'_1)^{k_1^0}(r'_2)^{k_2^0}}
\lambda_{2,1}^{j_1}(R'')\lambda_{2,2}^{j_2}(R'')|F(z^{*})|
\geq \nonumber
\\
\geq \ldots \geq
\frac{l_1^{k_1^0}(z^0)}{(r_1^{''})^{k_1^0}}|F^{(0,k_2^0)}(z_{0,k_2^0}^{*})|-
\frac{p_0k_2^0!l_1^{k_1^0}(z_1^0,z_2^0)l_2^{k_2^0}(z_1^0,z_2^0)}{(r_1^{''})^{k_1^0}(r'_1)^{k_1^0}(r'_2)^{k_2^0}}
\lambda_{2,1}^{j_1}(R'')\lambda_{2,2}^{j_2}(R'')|F(z^{*})|-\ldots -
\nonumber \\ -
\frac{p_0(k_1^0-2)!k_2^0!l_1^{k_1^0}(z_1^0,z_2^0)l_2^{k_2^0}(z_1^0,z_2^0)}{(r_1^{''})^2(r'_1)^{k_1^0}(r'_2)^{k_2^0}}
\lambda_{2,1}^{j_1}(R'')\lambda_{2,2}^{j_2}(R'')|F(z^{*})|-\nonumber \\ -
\frac{p_0(k_1^0-1)!k_2^0!l_1^{k_1^0}(z_1^0,z_2^0)l_2^{k_2^0}(z_1^0,z_2^0)}{r_1^{''}(r'_1)^{k_1^0}(r'_2)^{k_2^0}}
\lambda_{2,1}^{j_1}(R'')\lambda_{2,2}^{j_2}(R'')|F(z^{*})|
= \nonumber
\\
\!=\!\frac{l_1^{k_1^0}(z^0)}{(r_1^{''})^{k_1^0}}|F^{(0,k_2^0)}(z_{(0,k_2^0)}^{*})|\!-\!
\frac{p_0k_2^0!l_1^{k_1^0}(z_1^0,z_2^0)l_2^{k_2^0}(z_1^0,z_2^0)}{(r'_1)^{k_1^0}(r'_2)^{k_2^0}}
\lambda_{2,1}^{j_1}(R'')\lambda_{2,2}^{j_2}(R'')|F(z^{*})|
\sum_{j_1=1}^{k_1^0} \frac{(k_1^0-j_1)!}{(r_1^{''})^{j_1}} \geq
\nonumber \\
\!\geq\!
\frac{l_1^{k_1^0}(z^0)}{(r_1^{''})^{k_1^0}}\frac{l_2^{k_2^0}(z^0)}{(r_2^{''})^{k_2^0}}
|F(z_{(0,0)}^{*})|\!-\!
\frac{p_0k_2^0!l_1^{k_1^0}(z_1^0,z_2^0)l_2^{k_2^0}(z_1^0,z_2^0)}{(r'_1)^{k_1^0}(r'_2)^{k_2^0}}
\lambda_{2,1}^{k_1^0}(R'')\lambda_{2,2}^{k_2^0}(R'')|F(z^{*})|
\sum_{j_1=1}^{k_1^0} \frac{(k_1^0-j_1)!}{(r_1^{''})^{j_1}}
-\nonumber
\\
-\frac{p_0l_1^{k_1^0}(z_1^0,z_2^0)l_2^{k_2^0}(z_1^0,z_2^0)}{(r'_1)^{k_1^0}(r'_2)^{k_2^0}(r_1^{''})^{k_1^0}}
\lambda_{2,2}^{k_2^0}(R'')|F(z^{*})|
\sum_{j_2=1}^{k_2^0} \frac{(k_2^0-j_2)!}{(r_2^{''})^{j_2}}
=\nonumber \\
=\frac{l_1^{k_1^0}(z_1^0,z_2^0)l_2^{k_2^0}(z_1^0,z_2^0)}{(r''_1)^{k_1^0}(r''_2)^{k_2^0}}|F(z_{(0,0)}^{*})|-|F(z^{*})|(\tilde{b}_1+\tilde{b}_2),
\label{th4el}
\end{gather}
where 
\begin{gather}
  \tilde{b}_1=\frac{p_0k_2^0!l_1^{k_1^0}(z_1^0,z_2^0)l_2^{k_2^0}(z_1^0,z_2^0)}{(r'_1)^{k_1^0}(r'_2)^{k_2^0}}
\lambda_{2,1}^{k_1^0}(R'')\lambda_{2,2}^{k_2^0}(R'')
\sum_{j_1=1}^{k_1^0} \frac{(k_1^0-j_1)!}{(r_1^{''})^{j_1}}=
\nonumber \\
=p_0k_2^0!\frac{l_1^{k_1^0}(z_1^0,z_2^0)l_2^{k_2^0}(z_1^0,z_2^0)}{(r''_1)^{k_1^0}(r''_2)^{k_2^0}}
\frac{(r''_1)^{k_1^0}(r''_2)^{k_2^0}}{(r'_1)^{k_1^0}(r'_2)^{k_2^0}}
\lambda_{2,1}^{k_1^0}(R'')\lambda_{2,2}^{k_2^0}(R'')
\sum_{j_1=1}^{k_1^0} \frac{(k_1^0-j_1)!}{(r_1^{''})^{j_1}}\leq
\nonumber
\\
\leq
p_0N!\frac{l_1^{k_1^0}(z_1^0,z_2^0)l_2^{k_2^0}(z_1^0,z_2^0)}{(r''_1)^{k_1^0}(r''_2)^{k_2^0}}
\left(\frac{r''_1r''_2}{r'_1r'_2}\right)^{N}
\lambda_{2,1}^{N}(R'')\lambda_{2,2}^{N}(R'')
\sum_{j=1}^{N} \frac{(N-j)!}{(r_1^{''})^{j}}\!=\!\frac{l_1^{k_1^0}(z_1^0,z_2^0)l_2^{k_2^0}(z_1^0,z_2^0)}{(r''_1)^{k_1^0}(r''_2)^{k_2^0}}b_1,\nonumber \\ 
\tilde{b}_2=\frac{p_0}{(r_1^{''})^{k_1^0}}\frac{l_1^{k_1^0}(z_1^0,z_2^0)l_2^{k_2^0}(z_1^0,z_2^0)}{(r''_1)^{k_1^0}(r''_2)^{k_2^0}}
\lambda_{2,2}^{k_2^0}(R'')
\sum_{j_2=1}^{k_2^0} \frac{(k_2^0-j_2)!}{(r_2^{''})^{j_2}}\leq
\nonumber \\
\!\leq\! p_0\frac{l_1^{k_1^0}(z_1^0,z_2^0)l_2^{k_2^0}(z_1^0,z_2^0)}{(r''_1)^{k_1^0}(r''_2)^{k_2^0}}
\lambda_{2,2}^{N}(R'')
\sum_{j=1}^{N} \frac{(N-j)!}{(r_2^{''})^{j}}\max\left\{ \frac{1}{(r_1^{''})^N},1 \right\}\!=\!\frac{l_1^{k_1^0}(z_1^0,z_2^0)l_2^{k_2^0}(z_1^0,z_2^0)}{(r''_1)^{k_1^0}(r''_2)^{k_2^0}}b_2.
\label{th4thieen}
\end{gather}
The inequality \eqref{th4el} implies that
\begin{gather*}
|F^{(k^0_1,k^0_2)}(z_{(k^0_1,k^0_2)}^{*})|\geq \frac{l_1^{k_1^0}(z_1^0,z_2^0)l_2^{k_2^0}(z_1^0,z_2^0)}{(r''_1)^{k_1^0}(r''_2)^{k_2^0}}|F(z^{*})|\left( \frac{|F(z^{*}_{(0,0)})|}{|F(z^{*})|}-(b_1+b_2) \right).
\end{gather*}
In view of \eqref{th4f} we have that $\frac{|F(z^{*}_{(0,0)})|}{|F(z^{*})|} \geq p_{*}>b_1+b_2. $ Hence, applying \eqref{th4thi} and \eqref{th4s} to 
\eqref{th4thieen}, we deduce 
\begin{gather*}
|F^{(k^0_1,k^0_2)}(z_{(k^0_1,k^0_2)}^{*})|\geq \frac{l_1^{k_1^0}(z_1^0,z_2^0)l_2^{k_2^0}(z_1^0,z_2^0)}{(r''_1)^{k_1^0}(r''_2)^{k_2^0}}|F(z^{*})|( p_{*}-(b_1+b_2) )\geq
\nonumber \\
\geq
 \frac{l_1^{k_1^0}(z_1^0,z_2^0)l_2^{k_2^0}(z_1^0,z_2^0)}{(r''_1)^{k_1^0}(r''_2)^{k_2^0}}( p_{*}-(b_1+b_2) )\frac{|F^{(k^0_1,k^0_2)}(z^0)|(r'_1)^{k_1^0}(r'_2)^{k_2^0}}{k^0_1!k^0_2!l_1^{k_1^0}(z_1^0,z_2^0)l_2^{k_2^0}(z_1^0,z_2^0)}\geq
 \nonumber \\ \geq
 \left(\frac{r_1^{'}r_2^{'}}{r_1^{''}r_2^{''}}\right)^N(p_{*}-(b_1+b_2))
 \frac{|F^{(k^0_1,k^0_2)}(z_{(k^0_1,k^0_2)}^{*})|}{p_0(N!)^2}.
\end{gather*}
Therefore,
$p_{*}\leq p_0(N!)^2\left(\frac{r_1^{''}r_2^{''}}{r_1^{'}r_2^{'}}\right)^N+b_1+b_2,$
but it contradicts of choice $p_{*}=p_0(N!)^2\left(\frac{r_1^{''}r_2^{''}}{r_1^{'}r_2^{'}}\right)^N+b_1+b_2+1.$
The necessity is proved. 

Now we prove a sufficiency.
Let $z^0\in\mathbb{D}^2$ be an arbitrary point. We expand a function $F$ in power series in $\mathbb{D}^2(z^0,R)$
	\begin{equation}
	\label{th4suf1}
	F(z)=\sum_{k\geq\mathbf{0}}b_k(z-z^0)^k= \sum_{k_1\ge 0, k_2\geq 0}b_{k_1,k_2}(z_1-z_1^0)^{k_1}(z_2-z_2^0)^{k_2},
	\end{equation}
where $k=(k_1,k_2),$ $b_k=b_{k_1,k_2}=\frac{F^{(k_1,k_2)}(z^0_1,z_2^0)}{k_1!k_2!},$ $R=(r_1,r_2).$

Let $\mu(R,z^0,F)=\max\{|b_k|R^k\colon  \ k\geq \mathbf{0}\}=\max\{|b_{k_1,k_2}|r_1^{k_1}r_2^{k_2}\colon  \ k_1\geq 0, k_2\ge 0\} $ be a maximal term of series \eqref{th4suf1} and $\nu(R)=\nu(R,z^0,F)=(\nu_1^0(R),\nu_2^0(R))$ be a set of indices such that
	$$
	\mu(R,z^0,F)=|b_{\nu(R)}|R^{\nu(R)},$$
$$
	\|\nu(R)\|=\nu_1(R)+\nu_2(R)=\max \{k_1+k_2\colon  k_1\geq 0,k_2\geq0, \  |b_k|R^k=\mu(R,z^0,F) \}.
	$$
We apply Cauchy's inequality
$$
\forall R=(r_1,r_2), 0<r_j<1,  j\in\{1,2\} \colon \ \mu(R,z^0,F)\leq M(R,z^0,F).
$$
Choosing $R'$ and $R'',$ \ $0<r'_j<1,$  $1<r''_j<\beta,$ we conclude 
\begin{gather*}
 M(R'R,z^0,F)\leq \sum_{k\geq\mathbf{0}}|b_k|(R'R)^k\leq
\sum_{k\geq\mathbf{0}}\mu(R,z^0,F)(R')^k=\mu(R,z^0,F)\sum_{k\geq\mathbf{0}}(R')^k= \nonumber \\
=\prod_{j=1}^{2}\frac{1}{1-r'_j}\mu(R,z^0,F).
\end{gather*}
Besides, 
\begin{gather*}
  \ln \mu(R,z^0,F)=\ln \{|b_{\nu(R)}|R^{\nu(R)} \}=
  \ln\left\{|b_{\nu(R)}|(RR'')^{\nu(R)}\frac{1}{(R'')^{\nu(R)}}\right\}= \nonumber \\
   =\ln\{|b_{\nu(R)}|(RR'')^{\nu(R)}\}+\ln\left\{\frac{1}{(R'')^{\nu(R)}}\right\}\leq
    \ln \mu((R''R,z^0,F)-\|\nu(R) \| \ln \min \{r_1'',r_2''\}.
  \end{gather*}
  This implies that
  \begin{gather}
    \|\nu(R)\|\leq \frac{1}{\ln \min \{r_1'',r_2''\}}( \ln \mu(R''R,z^0,F)- \ln \mu(R,z^0,F))\leq
    \nonumber \\
    \leq
    \frac{1}{\ln \min \{r_1'',r_2''\}}\left( \ln M(R''R,z^0,F)- \ln((1-r_1')(1-r_2') M(R'R,z^0,F))\right)\leq
    \nonumber \\
    \leq
    \frac{1}{\ln \min \{r_1'',r_2''\}}\left( \ln M(R''R,z^0,F)- \ln M(R'R,z^0,F))\right)-
    \frac{\sum_{j=1}^{2}\ln(1-r'_j)}{\ln \min \{r_1'',r_2''\}}= \nonumber \\
    =\frac{1}{\ln \min \{r_1'',r_2''\}}
    \ln\frac{M(R''R,z^0,F)}{M(R'R,z^0,F)}- \frac{\sum_{j=1}^{2}\ln(1-R_j)}{\ln \min \{r_1'',r_2''\}}.
  \label{th4suf4}
  \end{gather}
 Put $R=\frac{\mathbf{1}}{\mathbf{L}(z^0)}.$ Now let $N(F,z^0,\mathbf{L})$ be a $\mathbf{L}$-index of the function $F$ in joint variables at point $z^0$ i. e. it is the least integer 
  for which inequality \eqref{net0} holds at point $z^0.$ Clearly that
  \begin{equation} \label{eq1ver}
N(F,z^0,\mathbf{L})\leq\nu\left(\frac{1}{\mathbf{L}(z^0)},z^0,F\right)=\nu(R,z^0,F).\end{equation}
  But 
  \begin{equation} \label{eq2ver}
M\left(\frac{R''}{\mathbf{L}(z^0)},z^0,F\right)\leq p_1(R',R'')M\left(\frac{R'}{\mathbf{L}(z^0)},z^0,F\right).
\end{equation}
  Therefore, from \eqref{th4suf4}, \eqref{eq1ver}, \eqref{eq2ver} we obtain that $\forall z^0 \in\mathbb{D}^2$
  $$
  N(F,z^0,\mathbf{L})\leq \frac{-\sum_{j=1}^{2}\ln(1-r'_j)}{\ln \min \{r_1'',r_2''\}}+\frac{\ln p_1(R',R'')}{\ln \min \{r_1'',r_2''\}}.
  $$
  This means that $F$ has bounded $\mathbf{L}$-index in joint variables.
\end{proof}

  \noindent{\bf 5. Theorem of Hayman for analytic in a bidisc function of bounded $\mathbf{L}$-index in joint variables.}
\begin{theorem} \label{petr5}
 Let $\mathbf{L} \in Q^2(\mathbb{D}^2)$. An analytic function $F$ in $\mathbb{D}^2$  has bounded  $\mathbf{L}$-index in joint variables if and only if there exist $p\in \mathbb{Z}_+$ and $c\in \mathbb{R}_+$ such that for each $z=(z_1,z_2) \in \mathbb{D}^2$ the next inequality holds
\begin{gather}\label{th5n1}
  \max \left\{ \frac{|F^{(j_1,j_2)}(z_1,z_2)|}{l_1^{j_1}(z_1,z_2)l_2^{j_2}(z_1,z_2)} \colon j_1+j_2=p+1 \right\} \leq c\max \left\{ \frac{|F^{(k_1,k_2)}(z_1,z_2)|}{l_1^{k_1}(z_1,z_2) l_2^{k_2}(z_1,z_2)} \colon k_1+k_2\leq p \right\}.
\end{gather}
\end{theorem}
\begin{proof}
  Let $N=N(F,\mathbf{L},\mathbb{D}^2)<+\infty.$
  The proof of the necessity implies from the definition of the boundedness of $\mathbf{L}$-index in joint variables with $p=N$ and $c=((N+1)!)^2.$
  We prove the sufficiency.
  Let \eqref{th5n1} holds, $z^0 \in \mathbb{D}^2$, $z \in \mathbb{T}^2\left(z^0, \frac{\boldsymbol{\beta}}{\mathbf{L}(z^0)}\right)$.
  For all $j=(j_1,j_2)\in \mathbb{Z}_+^2$, $j_1+j_2\leq p+1$ we have
  \begin{gather*}
    \frac{|F^{(j_1,j_2)}(z)|}{l_1^{j_1}(z_1^0,z_2^0)l_2^{j_2}(z_1^0,z_2^0)}\leq \frac{|F^{(j_1,j_2)}(z)|l_1^{j_1}(z_1,z_2)l_2^{j_2}(z_1,z_2)}{l_1^{j_1}(z_1^0,z_2^0)l_2^{j_2}(z_1^0,z_2^0)l_1^{j_1}(z_1,z_2)l_2^{j_2}(z_1,z_2)}\leq
    \lambda_{2,1}^{j_1}(\beta)\lambda_{2,2}^{j_2}(\beta) \times \\ \times 
    \frac{|F^{(j_1,j_2)}(z)|}{l_1^{j_1}(z_1,z_2)l_2^{j_2}(z_1,z_2)} 
      \leq     \lambda_{2,1}^{j_1}(\beta)\lambda_{2,2}^{j_2}(\beta)
    c\max \left\{ \frac{|F^{(k_1,k_2)}(z_1,z_2)|}{l_1^{k_1}(z_1,z_2)l_2^{k_2}(z_1,z_2)} \colon k_1+k_2\leq p \right\}= \nonumber \\
    =\lambda_{2,1}^{j_1}(\beta)\lambda_{2,2}^{j_2}(\beta)
    c\max \left\{ \frac{l_1^{k_1}(z_1^0,z_2^0)l_2^{k_2}(z_1^0,z_2^0)|F^{(k_1,k_2)}(z_1,z_2)|}{l_1^{k_1}(z_1^0,z_2^0)l_2^{k_2}(z_1^0,z_2^0)l_1^{k_1}(z_1,z_2)l_2^{k_2}(z_1,z_2)} \colon k_1+k_2 \leq p \right\}\leq \nonumber \\
    \leq  \lambda_{2,1}^{j_1}(\beta)\lambda_{2,2}^{j_2}(\beta)
    c\max \left\{\frac{1}{\lambda_{2,1}^{k_1}(\beta)\lambda_{2,2}^{k_2}(\beta)} \frac{|F^{(k_1,k_2)}(z_1,z_2)|}{l_1^{k_1}(z_1^0,z_2^0)l_2^{k_2}(z_1^0,z_2^0)} \colon k_1+k_2\leq p \right\}\leq
    \nonumber \\ \leq
   \max \{ \lambda_{2,1}^{j_1}(\beta)\lambda_{2,2}^{j_2}(\beta)\colon j_1+j_2\leq p+1\}
    c\max \left\{\frac{1}{\lambda_{2,1}^{k_1}(\beta)\lambda_{2,2}^{k_2}(\beta)} \colon k_1+k_2\leq p \right\}\times
    \nonumber \\
    \times
    \max \left\{\frac{|F^{(k_1,k_2)}(z_1,z_2)|}{l_1^{k_1}(z_1^0,z_2^0)l_2^{k_2}(z_1^0,z_2^0)} \colon k_1+k_2\leq p \right\}=B\cdot G(z),
  \end{gather*}
  where
  \begin{gather*}
      B=c \max \{ \lambda_{2,1}^{j_1}(\beta)\lambda_{2,2}^{j_2}(\beta)\colon j_1+j_2\leq p+1\}
    \max \left\{\frac{1}{\lambda_{2,1}^{k_1}(\beta)\lambda_{2,2}^{k_2}(\beta)}  \colon k_1+k_2\leq p \right\}, \\
    G(z)= \max \left\{\frac{|F^{(k_1,k_2)}(z_1,z_2)|}{l_1^{k_1}(z_1^0,z_2^0)l_2^{k_2}(z_1^0,z_2^0)} \colon k_1+k_2\leq p \right\}.
  \end{gather*}
  We choose $z^{(1)}=(z^{(1)}_1,z^{(1)}_2)\in \mathbb{T}^2\left(z^0,\frac{\mathbf{1}}{2\beta\mathbf{L}(z^0)} \right)$ arbitrarily
  and $z^{(2)}\!=\!(z^{(2)}_1,z^{(2)}_2)\!\in\! \mathbb{T}^2\left(z^0,\frac{\boldsymbol{\beta}}{\mathbf{L}(z^0)} \right)$ such that
 \begin{equation} \label{choicez2}
 |F(z^{(2)})|=\max \left\{|F(z)|\colon z \in \mathbb{T}^2\left(z^0,\frac{\boldsymbol{\beta}}{\mathbf{L}(z^0)} \right) \right\}.
\end{equation}
  We connect the points $z^{(1)}$ and $z^{(2)}$  with plane 
  $$ \alpha\colon \  \ z_2=k_2z_1+c_2
  $$
  $$ \frac{z_2-z_2^{(1)}}{z_2^{(2)}-z_2^{(1)}}=\frac{z_1-z_1^{(1)}}{z_1^{(2)}-z_1^{(1)}}, \ \
  k_2=\frac{z_2^{(2)}-z_2^{(1)}}{z_1^{(2)}-z_1^{(1)}}, \ \
  c_2=\frac{z_2^{(1)}z_1^{(2)}-z_1^{(1)}z_2^{(2)}}{z_1^{(2)}-z_1^{(1)}}.
  $$
  Let $\tilde{G(z_1)}=G(z)|_{\alpha} $ be a restriction of the function $G$ onto $\alpha$.
  All functions $F^{(k_1,k_2)}|_{\alpha}$ are analytic functions of $z_1$ in a unit disc and $\tilde{G}(z_1^{(1)})=G(z^{(1)})\neq 0$, otherwise all derivatives of $F$ at the point $z^{(1)}$ equal $0$ and $F\equiv 0$. That's why zeros of the function $\tilde{G}(z^{(1)})$ are isolated as zeros of a function of one variable. Therefore we can choose on $\alpha$ piecewise analytic curve
   $$ \gamma=\{z=(z_1(t),k_2z_1(t)+c_2),0\leq t\leq T \},
   $$
   which joins $z^{(1)}$ and $z^{(2)}$ so that $G(z(t))\neq0$ and its length does not exceed $\int_{0}^{T}|z_1^{'}(t)|dt\leq \frac{2\beta^2+1}{2\beta l_1(z^0)}$.
   Then
   \begin{gather*}
   \int_{0}^{T}|z_2^{'}(t)|dt=|k_2|\int_{0}^{T}|z_1^{'}(t)|dt\leq
   \left|\frac{z_2^{(2)}-z_2^{(1)}}{z_1^{(2)}-z_1^{(1)}} \right|
   \frac{2\beta^2+1}{2\beta l_1(z^0)}\leq
   \frac{2\beta^2+1}{2\beta l_2(z^0)}\frac{2\beta l_1(z^0)}{2\beta^2-1}
   \frac{2\beta^2+1}{2\beta l_1(z^0)}= \\
   =\frac{(2\beta^2+1)^2}{l_2(z^0)(2\beta^2-1)2\beta}.
   \end{gather*}
   Hence,
   \begin{gather}
     \int_{0}^{T}\sum_{i=1}^{2}l_i(z)|z_i^{'}(t)|dt\leq \lambda_{2,2}(\beta)l_2(z^0)\frac{(2\beta^2+1)^2}{l_2(z^0)(2\beta^2-1)2\beta}+
     \lambda_{2,1}(\beta)l_1(z^0)\frac{2\beta^2+1}{l_1(z^0)2\beta}= \nonumber 
     \\
     =\lambda_{2,2}(\beta)\frac{(2\beta^2+1)^2}{(2\beta^2-1)2\beta}+
     \lambda_{2,1}(\beta)\frac{2\beta^2+1}{2\beta}. \label{upperest}
   \end{gather}
   The upper estimate in \eqref{upperest} we denote by $S.$
   Without loss of generality we can assume that the function $z=z(t)$ is analytic on $[0,T]$.
   Then 
   for arbitrary $k\in\mathbb{Z}_+^2,  j\in\mathbb{Z}_+^2,  \|k \|\leq p, \|j \|\leq p, k\neq j$ either
   $$\frac{|F^{(k_1,k_2)}(z_1(t),z_2(t))|}{l_1^{k_1}(z_1^0,z_2^0)l_2^{k_2}(z_1^0,z_2^0)}\equiv \frac{|F^{(j_1,j_2)}(z_1(t),z_2(t))|}{l_1^{j_1}(z_1^0,z_2^0)l_2^{j_2}(z_1^0,z_2^0)}
   $$
   or the equality
   $$\frac{|F^{(k_1,k_2)}(z_1(t),z_2(t))|}{l_1^{k_1}(z_1^0,z_2^0)l_2^{k_2}(z_1^0,z_2^0)}=
   \frac{|F^{(j_1,j_2)}(z_1(t),z_2(t))|}{l_1^{j_1}(z_1^0,z_2^0)l_2^{j_2}(z_1^0,z_2^0)}
   $$
   holds only for a finite set of points $t_k\in[0;T]$.
   Hence we can partition the segment $[0;T]$ onto a finite number of segments such that on each of them the equality
   $$G(z(t))=\frac{|F^{(j_1,j_2)}(z_1(t),z_2(t))|}{l_1^{j_1}(z_1^0,z_2^0)l_2^{j_2}(z_1^0,z_2^0)}
   $$
   holds with some $j_1+j_2\leq p$.
The function  $G(z(t))$ is a continuously differentiable function with the exception, perhaps, of a finite set of points. Using the inequality $\frac{d}{dx} \left|\varphi(x)\right| \leq \left|\frac{d}{dx}\varphi(x)\right|,$ which holds for complex-valued functions of real argument outside a countable set of points,  in view of \eqref{net1},  we have
   \begin{gather*}
     \frac{d}{dt}G(z(t))\leq \max \left\{\frac{1}{l_1^{j_1}(z_1^0,z_2^0)l_2^{j_2}(z_1^0,z_2^0)}
     \left|\frac{d}{dt}F^{(j_1,j_2)}(z_1(t),z_2(t))\right| \colon j_1+j_2\leq p \right\}\leq
     \\ \leq
   \max \left\{\frac{|F^{(j_1+1,j_2)}(z_1(t),z_2(t))|\cdot |z_1^{'}(t)|}{l_1^{j_1}(z_1^0,z_2^0)l_2^{j_2}(z_1^0,z_2^0)}
     +\frac{|F^{(j_1,j_2+1)}(z_1(t),z_2(t))|\cdot  |z_2^{'}(t)|}{l_1^{j_1}(z_1^0,z_2^0)l_2^{j_2}(z_1^0,z_2^0)}
      \colon j_1+j_2\leq p \right\}=
      \\ = \max \left\{|F^{(j_1+1,j_2)}(z(t))|\frac{|z_1^{'}(t)|l_1(z^0)}{l_1^{j_1+1}(z^0)l_2^{j_2}(z^0)}
     +|F^{(j_1,j_2+1)}(z(t))|\frac{|z_2^{'}(t)|l_2(z^0)}{l_1^{j_1}(z^0)l_2^{j_2+1}(z^0)}
      \colon j_1+j_2 \leq p \right\}\leq
      \\ \leq
      (|z_1^{'}(t)|l_1(z^0)+|z_2^{'}(t)|l_2(z^0))\max\left\{
      \frac{|F^{(j_1,j_2)}(z_1(t),z_2(t))|}{l_1^{j_1}(z_1^0,z_2^0)l_2^{j_2}(z_1^0,z_2^0)}\colon j_1+j_2\leq p+1 \right\} \leq
      \\ \leq \left(\sum_{i=1}^{2}l_i(z^0)|z_i^{'}(t)|\right) BG(z(t)).
         \end{gather*}
         Therefore, \eqref{upperest} yields 
         \begin{gather*}
           \left|\ln\frac{G(z^{(2)})}{G(z^{(1)})} \right|=\left| \int_{0}^{T}\frac{1}{G(z(t))}\frac{d}{dt}G(z(t)) \right| \leq
           B\int_{0}^{T}\sum_{i=1}^{2}l_i(z^0)|z_i^{'}(t)|dt\leq B\cdot S.
         \end{gather*}
  Using \eqref{choicez2} we deduce 
   \begin{gather*}
     \max \left\{|F(z)|\colon z\in \mathbb{T}^2\left(z^0, \frac{\boldsymbol{\beta}}{\mathbf{L}(z^0)}\right) \right\}=|F(z^{(2)})|\leq G(z^{(2)})\leq G(z^{(1)})\cdot\exp{BS}.
   \end{gather*}
   Since $z^{(1)}\in \mathbb{T}^2(z^0, \frac{\mathbf{1}}{2\beta\mathbf{L}(z^0)})$ then for all $j\in\mathbb{Z}_+^2$ the Cauchy inequality holds:
   \begin{gather*}
        \frac{|F^{(j)}(z^{(1)})|}{l_1^{j_1}(z_1^0,z_2^0)l_2^{j_2}(z_1^0,z_2^0)}\leq
        j_1!j_2!(2\beta)^{j_1+j_2}M\left(\frac{\mathbf{1}}{2\beta\mathbf{L}(z^0)},z^0,F\right).
   \end{gather*}
   Therefore, $G(z^{(1)})\leq(p!)^2(2\beta)^{2p}M(\frac{1}{2\beta\mathbf{L}(z^0)},z^0,F)$ and 
   \begin{gather*}
   \max \left\{|F(z)|\colon z\in \mathbb{T}^2\left(z^0, \frac{\boldsymbol{\beta}}{\mathbf{L}(z^0)}\right) \right\}\!\leq\! e^{BS}(p!)^2(2\beta)^{2p}\max \left\{|F(z)|\colon z\in \mathbb{T}^2\left(z^0, \frac{\mathbf{1}}{2\beta\mathbf{L(z^0)}}\right) \right\}.
   \end{gather*}
   Hence, by Theorem \ref{petr4} $F$ is a function of bounded $\mathbf{L}$-index in joint variables.
   \end{proof}
\begin{theorem}
 Let $\beta>1$, $\mathbf{L} \in Q^2(\mathbb{D}^2)$. An analytic function $F$ in $\mathbb{D}^2$  has bounded  $\mathbf{L}$-index in joint variables if and only if there exist $c\in (0;+\infty)$ and $N\in \mathbb{N}$ such that for each $z \in \mathbb{D}^2$ the next inequality holds
\begin{gather}\label{th6main}
  \sum_{k_1+k_2=0}^{N}\frac{F^{(k_1,k_2)}(z_1,z_2)}{k_1!k_2!l_1^{k_1}(z_1,z_2)
  l_2^{k_2}(z_1,z_2)}\geq c\sum_{k_1+k_2=N+1}^{\infty}\frac{F^{(k_1,k_2)}(z_1,z_2)}{k_1!k_2!l_1^{k_1}(z_1,z_2)
  l_2^{k_2}(z_1,z_2)}.
\end{gather}
\end{theorem}
\begin{proof}
  Let $\frac{1}{\beta}<\theta_j<1,$ $j\in\{1,2\}.$ If $F$ has bounded $\mathbf{L}$-index in joint variables then by Theorem \ref{petr2} $F$ has bounded $\tilde{\mathbf{L}}$-index in joint variables, where $\tilde{\mathbf{L}}=(\tilde{l}_1(z),\tilde{l}_2(z))$, $\tilde{l}_j(z)=\theta_jl_j(z)$, $j\in\{1,2\}$.
  Therefore,
  \begin{gather*}
    \max\left\{\frac{|F^{(k_1,k_2)}(z_1,z_2)|}{k_1!k_2!l_1^{k_1}(z_1,z_2)
  l_2^{k_2}(z_1,z_2)}\colon 0\leq k_1+k_2\leq N(F,\tilde{\mathbf{L}}, \mathbb{D}^2) \right\}=  \nonumber
  \\ =\max\left\{\frac{\theta_1^{k_1}\theta_2^{k_2}|F^{(k_1,k_2)}(z_1,z_2)|}
  {k_1!k_2!\tilde{l}_1^{k_1}(z_1,z_2)
  \tilde{l}_2^{k_2}(z_1,z_2)}\colon 0\leq k_1+k_2\leq N(F,\tilde{\mathbf{L}}, \mathbb{D}^2) \right\}\geq
  \nonumber \\ \geq
  (\theta_1\theta_2)^{N(F,\tilde{\mathbf{L}}, \mathbb{D}^2)}\max\left\{\frac{|F^{(k_1,k_2)}(z_1,z_2)|}
  {k_1!k_2!\tilde{l}_1^{k_1}(z_1,z_2)
  \tilde{l}_2^{k_2}(z_1,z_2)}\colon 0\leq k_1+k_2\leq N(F,\tilde{\mathbf{L}}, \mathbb{D}^2) \right\}\geq
  \nonumber \\ \geq
    (\theta_1\theta_2)^{N(F,\tilde{\mathbf{L}}, \mathbb{D}^2)}\frac{|F^{(j_1,j_2)}(z_1,z_2)|}
  {j_1!j_2!\tilde{l}_1^{j_1}(z_1,z_2)
  \tilde{l}_2^{j_2}(z_1,z_2)}=
  \nonumber \\
  = \theta_1^{N(F,\tilde{\mathbf{L}}, \mathbb{D}^2)-j_1}\theta_2^{N(F,\tilde{\mathbf{L}}, \mathbb{D}^2)-j_2}\frac{|F^{(j_1,j_2)}(z_1,z_2)|}
  {j_1!j_2!l_1^{j_1}(z_1,z_2)
  l_2^{j_2}(z_1,z_2)}
  \end{gather*}
  for all $j_1\geq 0, j_2\geq 0$ and
 \begin{gather*}
 \sum_{j_1+j_2=N(F,\tilde{\mathbf{L}}, \mathbb{D}^2)+1}^{\infty}\frac{|F^{(j_1,j_2)}(z_1,z_2)|}{j_1!j_2!l_1^{j_1}(z_1,z_2)
 l_2^{j_2}(z_1,z_2)} \leq
 \nonumber \\ \leq
 \max\left\{\frac{|F^{(k_1,k_2)}(z_1,z_2)|}{k_1!k_2!l_1^{k_1}(z_1,z_2)
  l_2^{k_2}(z_1,z_2)}\colon 0\leq k_1+k_2\leq N(F,\tilde{\mathbf{L}}, \mathbb{D}^2) \right\}\times\nonumber \\\times\sum_{j_1+j_2=N(F,\tilde{\mathbf{L}}, \mathbb{D}^2)+1}^{\infty}\theta_1^{j_1-N(F,\tilde{\mathbf{L}}, \mathbb{D}^2)}\theta_2^{j_2-N(F,\tilde{\mathbf{L}}, \mathbb{D}^2)}= \nonumber \\ =\frac{\theta_1\theta_2}{(1-\theta_1)(1-\theta_2)} \max\left\{\frac{|F^{(k_1,k_2)}(z_1,z_2)|}{k_1!k_2!l_1^{k_1}(z_1,z_2)
  l_2^{k_2}(z_1,z_2)}\colon 0\leq k_1+k_2\leq N(F,\tilde{\mathbf{L}}, \mathbb{D}^2) \right\}\leq \nonumber \\
  \leq \frac{\theta_1\theta_2}{(1-\theta_1)(1-\theta_2)}
  \sum_{k_1+k_2=0}^{N(F,\tilde{\mathbf{L}}, \mathbb{D}^2)}\frac{|F^{(k_1,k_2)}(z_1,z_2)|}{k_1!k_2!l_1^{k_1}(z_1,z_2)
  l_2^{k_2}(z_1,z_2)}.
 \end{gather*}
 Hence, we obtain \eqref{th6main} with $N=N(F,\tilde{\mathbf{L}}, \mathbb{D}^2)$ and $c=\frac{\theta_1\theta_2}{(1-\theta_1)(1-\theta_2)}.$
 On the contrary, inequality \eqref{th6main} imply 
 \begin{gather*}
   \max\left\{\frac{|F^{(j_1,j_2)}(z_1,z_2)|}{j_1!j_2!l_1^{j_1}(z_1,z_2)
 l_2^{j_2}(z_1,z_2)}\colon  j_1+j_2=N+1 \right\}\leq \nonumber \\
 \leq \sum_{k_1+k_2=N+1}^{\infty}\frac{F^{(k_1,k_2)}(z_1,z_2)}{k_1!k_2!l_1^{k_1}(z_1,z_2)
  l_2^{k_2}(z_1,z_2)}
  \leq\frac{1}{c}\sum_{k_1+k_2=0}^{N}\frac{|F^{(k_1,k_2)}(z_1,z_2)|}{k_1!k_2!l_1^{k_1}(z_1,z_2)
  l_2^{k_2}(z_1,z_2)}\leq \nonumber \\
  \leq \frac{(N+1)N}{2c}\max \left\{ \frac{|F^{(k_1,k_2)}(z_1,z_2)|}{k_1!k_2!l_1^{k_1}(z_1,z_2)
  l_2^{k_2}(z_1,z_2)}\colon 0\leq k_1+k_2\le N \right\}
  \end{gather*}
  and by Theorem \ref{petr5} $F$ is of bounded $\mathbf{L}$-index in joint variables.
\end{proof}
  \noindent{\bf 6. Some property of power expansion of analytic in a bidisc function of bounded $\mathbf{L}$-index in joint variables.}
Let $(z_1^0,z_2^0)\in \mathbb{D}^2$. We develop an analytic in $\mathbb{D}^2$ function $F(z_1,z_2)$ in the power series written in a diagonal form
\begin{gather}\label{th7main}
  F(z_1,z_2)=\sum_{k_1+k_2=0}^{\infty}p_{k_1+k_2}((z_1-z_1^0),(z_2-z_2^0))=
    \sum_{k=0}^{\infty}\sum_{j_1+j_2=k}b_{j_1,j_2}(z_1-z_1^0)^{j_1}(z_2-z_2^0)^{j_2},
\end{gather}
where $p_{k}$ are homogeneous polynomials of $k$-th power.
The polynomial $p_{k_0}, k_0\in \mathbb{Z}_+,$ is called a main polynomial in the power expansion \eqref{th7main} on $\mathbb{T}^2(z^0,R)$ if for every $z\in \mathbb{T}^2(z^0,R)$ the next inequality holds:
\begin{gather*}
  |\sum_{k_1+k_2\neq k^0}p_{k_1+k_2}((z_1-z_1^0),(z_2-z_2^0))|\leq \frac{1}{2}
  \max\{|b_{j_1,j_2}|r_1^{j_1}r_2^{j_2}\colon j_1+j_2=k^0 \}, \\
   \text{ where } b_{j_1,j_2}=\frac{F^{(j_1,j_2)}(z_1^0,z_2^0)}{j_1!j_2!}. \nonumber
\end{gather*}
\begin{theorem}
 Let $\beta>1$, $\mathbf{L} \in Q^2(\mathbb{D}^2)$. An analytic function $F$ in $\mathbb{D}^2$  has bounded  $\mathbf{L}$-index in joint variables if and only if there exists $p\in \mathbb{Z}_+$  that for all $d\in (0;\beta]$ there exists $\eta(d)\in (0;d)$ such that for each $z^0 \in \mathbb{D}^2$ and some $r=r(d,z^0) \in(\eta (d),d),$ $k^0=k^0(d,z^0)\leq p$ the polynomial $p_{k^0}$ is the main polynomial in the series \eqref{th7main} on $\mathbb{T}^2(z^0,\frac{R}{\mathbf{L}(z^0)})$ with $R=(r,r).$ 
 \end{theorem}
\begin{proof}
  Let $F$ be of bounded $\mathbf{L}$-index in joint variables with $N=N(F,\mathbf{L},\mathbb{D}^2)<+\infty$ and 
  $n_0$ be $\mathbf{L}$-index in joint variables at a point $z^0\in\mathbb{D}^2.$ Then for each $z^0 \in \mathbb{D}^2$ \  $n_0\leq N$.
  We put
  \begin{gather*}
    a_{j_1,j_2}^{*}=\frac{|b_{j_1,j_2}|}{l_1^{j_1}(z^0)l_2^{j_2}(z^0)}=
    \frac{|F^{(j_1,j_2)}(z_1^0,z_2^0)|}{j_1!j_2!l_1^{j_1}(z^0)l_2^{j_2}(z^0)},
    \\
    a_k=\max\{a_{j_1,j_2}^{*} \colon j_1+j_2=k \}, \\
c=2((N+1)^3+6(N+3)!). \nonumber 
  \end{gather*}
  Let $d\in (0;\beta]$ be an arbitrary number. We put $r_m=\frac{d}{(d+1)c^m}, m\in\mathbb{Z}_+$ and denote  $$\mu_m=\max\{a_kr_m^k \colon k\in \mathbb{Z}_+\},\ s_m=\min\{k \colon a_kr_m^k=\mu_m \}.$$ Since $z^0$ is a fixed point the inequality $a_{k_1,k_2}^{*}\leq \max\{a_{j_1,j_2}^{*} \colon j_1+j_2\leq n_0 \}$ is valid for all $(k_2,k_2) \in \mathbb{Z}_+^2.$ Then $a_k\leq a_{n_0}$ for all $k \in \mathbb{Z}_+$.
  Hence, for all $k>n_0$ in view of $r_0<1$ we have
  \begin{gather*}
    a_kr_0^k<a_{n_0}r_0^{n_0}.
  \end{gather*}
  This implies $s_0\leq n_0$.
  Since $cr_m=r_{m-1}$, we obtain that for each $k>s_{m-1}$ 
  \begin{gather}
    a_{s_{m-1}}r_m^{s_{m-1}}=a_{s_{m-1}}r_{m-1}^{s_{m-1}}c^{-s_{m-1}}\geq
    a_kr_{m-1}^kc^{-s_{m-1}}=a_kr_{m}^kc^{k-s_{m-1}}\geq ca_kr_m^k.
  \label{th7_127}
  \end{gather}
  From \eqref{th7_127} it follows that $s_m\leq s_{m-1}$ for all $m \in \mathbb{N}$. Thus, we can rewrite 
  $$
  \mu_0=\max\{a_kr_0^k \colon k\leq n_0\},\ \mu_m=\max\{a_kr_m^k \colon k\leq s_{m-1}\}.
  $$
  We denote
  \begin{gather*}
    \mu_0^{*}=\max\{a_kr_0^k \colon s_0\neq k\leq n_0\}, \\
    \mu_m^{*}=\max\{a_kr_m^k \colon s_m\neq k\leq s_{m-1}\}, \\
    s_0^{*}=\min\{k \colon k\neq s_0, a_kr_0^k=\mu_0^{*} \}, \\
    s_m^{*}=\min\{k \colon k\neq s_m, a_kr_m^k=\mu_m^{*} \}, m \in \mathbb{N}
  \end{gather*}
  and we will show that there exists $m_0 \in\mathbb{Z}_+$ such, that
  \begin{gather}\label{th7_128}
   \frac{\mu_{m_0}^{*}}{\mu_{m_0}} \leq \frac{1}{c}.
  \end{gather}
 Suppose that for all $m\in \mathbb{Z}_+$ the next inequality holds
  \begin{gather}\label{th7_129}
    \frac{\mu_m^{*}}{\mu_m} > \frac{1}{c}.
  \end{gather}
  If $s_m^{*}<s_m$ $(s_m^{*}\neq s_m$ in view of definition) then we have 
  \begin{gather*}
    a_{s_m^{*}}r_{m+1}^{s_m^{*}}=\frac{a_{s_m^{*}}r_{m}^{s_m^{*}}}{c^{s_m^{*}}}=
    \frac{\mu_m^{*}}{c^{s_m^{*}}}>\frac{\mu_m}{c^{s_m^{*}+1}}=
    \frac{a_{s_m}r_{m}^{s_m}}{c^{s_m^{*}+1}}=
    \frac{a_{s_m}r_{m+1}^{s_m}}{c^{s_m^{*}+1-s_m}}\geq a_{s_m}r_{m+1}^{s_m},
      \end{gather*}
      and for all $k>s_m^{*}, k\neq s_m$, similarly, 
      \begin{gather*}
  a_{s_m^{*}}r_{m+1}^{s_m^{*}}=\frac{a_{s_m^{*}}r_{m}^{s_m^{*}}}{c^{s_m^{*}}}\geq
      \frac{a_{k}r_{m}^{k}}{c^{s_m^{*}}}\geq \frac{a_{k}r_{m}^{k}}{c^{k-1}}=
      \frac{ca_{k}r_{m}^{k}}{c^{k}}=ca_{k}r_{m+1}^{k},
      \end{gather*}
      i.e. $a_{s_m^{*}}r_{m+1}^{s_m^{*}}>a_{k}r_{m+1}^{k}$ for all $k>s_m^{*}$.
      Hence,
      \begin{gather}\label{th7_130}
        s_{m+1}\leq s_m^{*}\leq s_m-1.
      \end{gather}
      On the contrary, if $s_m<s_m^{*}\leq s_{m-1}$ then the equality $s_{m+1}=s_m$ may hold.
      But in this case the inequalities $s_{m+1}^{*}\leq s_m$ and $s_m^{*}\neq s_{m+1}$ imply that $$s_{m+1}^{*}< s_{m+1} \ \ (s_{m+1}^{*}\neq s_{m+1}).$$
      Instead of \eqref{th7_130} we have the inequality
      $$ s_{m+2}\leq s_{m+1}^{*} \leq s_{m+1}-1=s_m-1.
      $$
      Hence, if for all $m\in \mathbb{Z}_+$ estimate \eqref{th7_129} is true then for all $m\in \mathbb{Z}_+$ either inequality
      $s_{m+1}\leq s_m-1$ or $s_{m+2}\leq s_m-1$ holds, i.e. $s_{m+2}\leq s_m-1,$ because $s_{m+2}\leq s_{m+1}.$ 
      It implies that
      \begin{gather*}
        s_m\leq s_{m-2}-1\leq \ldots \leq s_{m-2\lfloor\frac{m}{2}\rfloor}-\left\lfloor\frac{m}{2}\right\rfloor\leq s_0-\left\lfloor\frac{m}{2}\right\rfloor\leq
        n_0-\left\lfloor\frac{m}{2}\right\rfloor\leq N-\left\lfloor\frac{m}{2}\right\rfloor,
      \end{gather*}
      i.e. $s_m<0$ if only $m>2N+1$, which is impossible.
      Therefore, there exists $m_0\leq 2N+1$ such that \eqref{th7_128} holds. 
      We put 
      $$
        r=r_{m_0}, \ \eta (d)=\frac{d}{(d+1)c^{2(N+1)}}, \ p=N \ \text{ and } \ k_0=s_{m_0}.
      $$
      Then for all $j_1+j_2\neq k_0=s_{m_0}$ on $\mathbb{T}^2(z^0, \frac{r}{\mathbf{L}(z^0)})$ in view \eqref{th7_128} we have
      \begin{gather*}
        |b_{j_1,j_2}||z_1-z_1^0|^{j_1}|z_2-z_2^0|^{j_2}=a_{j_1,j_2}^{*}r^{j_1+j_2} \leq a_{j_1+j_2}r^{j_1+j_2} 
        \leq \mu^*_{m_0} \leq \frac{1}{c} \mu_{m_0} \le \\ 
        \leq \frac{1}{c}a_{s_{m_0}}r_{m_0}^{s_{m_0}}=
        \frac{1}{c}a_{k_0}r^{k_0}.
      \end{gather*}
    Thus, on $\mathbb{T}^2(z^0, \frac{r}{\mathbf{L}(z^0)})$ we obtain 
      \begin{gather}
      \left|\sum_{j_1+j_2\neq k_0}b_{j_1,j_2}(z_1-z_1^0)^{j_1}(z_2-z_2^0)^{j_2}\right| \leq
      \sum_{j_1+j_2\neq k_0}a_{j_1,j_2}^{*}r^{j_1+j_2}\leq
      \sum_{k=0, k\neq k_0}^{\infty}a_k(k+1)^2r^k=\nonumber \\
      =\sum_{k=0, \ k\neq s_{m_0}}^{s_{m_0-1}}a_k(k+1)^2r^k+\sum_{k=s_{m_0-1}+1}^{\infty}a_k(k+1)^2r^k.
        \label{th7_131}
      \end{gather}
      We will estimate two sums in \eqref{th7_131}. 
     From \eqref{th7_128} it follows that  $ \mu^{*}_{m_0}\leq\frac{1}{c}\mu_{m_0}$ or 
     $ \max\{a_kr^k_{m_0}\colon k\neq s_{m_0}, \ k\le s_{m_0-1}\}\leq \frac{1}{c}\max\{a_kr_{m_0}^k: k\neq s_{m_0}, k\leq s_{m_0-1} \},$ i. e.  $a_kr^k\leq \frac{1}{c}a_{k_0}r^{k_0}.$ Then 
  \begin{gather}
        \sum_{k=0, \ k\ne s_{m_0}}^{s_{m_0-1}}a_k(k+1)^2r^k\leq \frac{a_{k_0}r^{k_0}}{c}\sum_{k=0}^{N}(k+1)^2\leq \frac{a_{k_0}r^{k_0}}{c}(N+1)^3.
        \label{th7_132}
      \end{gather}
      For each $k$ the inequality $a_kr_{m_0-1}^k\leq \mu_{m_0-1}$ holds and, hence,
      \begin{equation} \label{equ1}
a_kr_{m_0}^k=\frac{a_kr_{m_0-1}^k}{c^k}\leq \frac{\mu_{m_0-1}}{c^k}.
\end{equation}
Using \eqref{equ1} and  \eqref{th7_128} we deduce 
     \begin{gather}
     \sum_{k=s_{m_0-1}+1}^{\infty}a_k(k+1)^2r^k\leq 
\mu_{m_0-1}\sum_{k=s_{m_0-1}+1}^{\infty}(k+1)^2\frac{1}{c^k}
     =a_{s_{m_0-1}}r_{m_0-1}^{s_{m_0-1}}
     \sum_{k=s_{m_0-1}+1}^{\infty}(k+1)^2\frac{1}{c^k}=\nonumber \\
     =a_{s_{m_0-1}}\frac{r_{m_0-1}^{s_{m_0-1}}}{c^{s_{m_0-1}}}c^{s_{m_0-1}}
     \sum_{k=s_{m_0-1}+1}^{\infty}(k+1)^2\frac{1}{c^k}
      \!\leq\!
     a_{s_{m_0-1}}r_{m_0}^{s_{m_0-1}}c^{s_{m_0-1}}
     \sum_{k=s_{m_0-1}+1}^{\infty}(k+1)(k+2)\frac{1}{c^k}\leq \nonumber \\
     \leq \frac{a_{s_{m_0}}r^{s_{m_0}}}{c}c^{s_{m_0-1}}
     \bigg(\sum_{k=s_{m_0-1}+1}^{\infty}x^{k+2}\bigg)^{(2)}\bigg|_{x=\frac{1}{c}}=
          \frac{a_{k_0}r^{k_0}}{c}c^{s_{m_0-1}}\left(\frac{x^{s_{m_0-1}+3}}{1-x} \right)^{(2)}\bigg|_{x=\frac{1}{c}}= \nonumber \\
          \nonumber \!=\!
          \frac{a_{k_0}r^{k_0}}{c}c^{s_{m_0-1}}\bigg( \frac{(s_{m_0-1}+3)(s_{m_0-1}+2)x^{s_{m_0-1}+1}}{1-x}+
          \frac{2(s_{m_0-1}+3)x^{s_{m_0-1}+2}}{(1-x)^2}\!+\!
          \frac{2x^{s_{m_0-1}+3}}{(1-x)^3} \bigg)\bigg|_{x=\frac{1}{c}}\leq \nonumber \\
          \leq \frac{a_{k_0}r^{k_0}}{c}c^{s_{m_0-1}}2(s_{m_0-1}+3)(s_{m_0-1}+2) \sum_{j=0}^2 \frac{x^{s_{m_0-1}+1+j}}{(1-x)^{1+j}}
          \bigg|_{x=\frac{1}{c}}\leq \nonumber \\
          \le \frac{a_{k_0}r^{k_0}}{c}2(N+3)!
          \sum_{j=0}^2\frac{1}{(c-1)^{1+j}}\leq 
          \frac{a_{k_0}r^{k_0}}{c}6(N+3)!,
      \label{th7_133}
     \end{gather}
     because $c\geq2$.
     Hence, from \eqref{th7_132} and \eqref{th7_133} we obtain
     \begin{gather*}
     \left|\sum_{j_1+j_2\neq k_0}b_{j_1,j_2}(z_1-z_1^0)^{j_1}(z_2-z_2^0)^{j_2}\right| \leq
      \nonumber \\ \leq 
      \frac{a_{k_0}r^{k_0}}{c}(N+1)^3+6\frac{a_{k_0}r^{k_0}}{c}(N+3)!=
      \frac{a_{k_0}r^{k_0}}{c}((N+1)^3+6(N+3)!)=\frac{1}{2}a_{k_0}r^{k_0},
      \end{gather*}
         Therefore, the polynomial $p_{k^0}$ is the main polynomial in the series \eqref{th7main} on $\mathbb{T}^2(z^0,\frac{R}{\mathbf{L}(z^0)}).$
     
     The necessity is proved.
     
Now we prove the sufficiency.
Suppose that there exist $p\in\mathbb{Z}_+$ and  $\eta\in(0,d)$ such that for each $z^0\in\mathbb{D}^2$ and $d=1<\beta$ with some $r=r(1,z^0)\in(\eta(1),1)$ and $k_0=k_0(1,z^0)\leq p$ the polynomial $p_{k^0}$ is the main polynomial in the series \eqref{th7main} on $\mathbb{T}^2(z^0,\frac{R}{\mathbf{L}(z^0)}).$
Then 
\begin{gather}
  \left|\sum_{j_1+j_2\neq k_0}b_{j_1,j_2}(z_1-z_1^0)^{j_1}(z_2-z_2^0)^{j_2}\right|=
  \left|F(z_1,z_2)-\sum_{j_1+j_2= k_0}b_{j_1,j_2}(z_1-z_1^0)^{j_1}(z_2-z_2^0)^{j_2}\right|\leq \nonumber \\ \leq \frac{a_{k_0}r^{k_0}}{2}.
  \label{th7_1333}
  \end{gather}
Using \eqref{th7_1333} and Cauchy's inequality we have:
  \begin{gather*}
    |b_{j_1,j_2}(z_1-z_1^0)^{j_1}(z_2-z_2^0)^{j_2}|=a_{j_1,j_2}^{*}r^{j_1+j_2}\leq \frac{a_{k_0}r^{k_0}}{2}
  \end{gather*}
  for all $j_1,j_2\in\mathbb{Z}_+,$ i.e. for all $k\neq k_0$ 
  \begin{gather}\label{th7_134}
    a_kr^k\leq  \frac{a_{k_0}r^{k_0}}{2}.
  \end{gather}
  Suppose that $F$ is not a function of bounded $\mathbf{L}$-index in joint variables. Then in view of theorem \ref{petr5} for all $p_1\in\mathbb{Z}_+$ and $c\geq 1$ there exists $z^0\in\mathbb{D}^2$ such that the next inequality holds:
  \begin{gather*}
    \max \left\{
    \frac{|F^{(j_1,j_2)}(z_1^0,z_2^0)|}{l_1^{j_1}(z_1^0,z_2^0)l_2^{j_2}(z_1^0,z_2^0)}
    \colon j_1+j_2=p_1+1
     \right\}\!
     >\!c\max \left\{
    \frac{|F^{(k_1,k_2)}(z_1^0,z_2^0)|}{l_1^{k_1}(z_1^0,z_2^0)l_2^{k_2}(z_1^0,z_2^0)}
    \colon k_1+k_2\leq p_1
     \right\}.
  \end{gather*}
  We put $p_1=p$ and $c=\left(\frac{(p+1)!}{(\eta(1))^{p+1}}\right)^2.$ Then for this $z^0(p_1,c)$ we obtain:
  \begin{gather*}
    \max \left\{
    \frac{|F^{(j_1,j_2)}(z_1^0,z_2^0)|}{j_1!j_2!l_1^{j_1}(z_1^0,z_2^0)l_2^{j_2}(z_1^0,z_2^0)}
    \colon j_1+j_2=p+1
     \right\}> \\
     >\frac{1}{(\eta(1))^{p+1}}\max \left\{
    \frac{|F^{(k_1,k_2)}(z_1^0,z_2^0)|}{k_1!k_2!l_1^{k_1}(z_1^0,z_2^0)l_2^{k_2}(z_1^0,z_2^0)}
    \colon k_1+k_2\leq p
     \right\},
  \end{gather*}
  i.e. $a_{p+1}>\frac{a_{k_0}}{(\eta(1))^{p+1}}$ and, hence,
  $a_{p+1}r^{p+1}>\frac{a_{k_0}r^{p+1}}{(\eta(1))^{p+1}}\geq a_{k_0}r^{k_0}.$ This is a contradiction with \eqref{th7_134}.
  Therefore, $F$ is of bounded $\mathbf{L}$-index in joint variables.
\end{proof}

It is easy to see that in the poof of sufficiency the radii $R=(r,r)$ of skeleton  $\mathbb{T}^2(z^0,\frac{R}{\mathbf{L}(z^0)})$ can be replaced 
by the  radii $R=(r_1,r_2),$ where $r_1$ is not necessarily equal to $r_2.$ Thus, the following theorem is true.

\begin{theorem} \label{petr8}
 Let $\beta>1$, $\mathbf{L} \in Q^2(\mathbb{D}^2)$. If there exists $p\in \mathbb{Z}_+$  that for all $d\in (0;\beta]$ there exists $\eta(d)\in (0;d)$ such that for each $z^0 \in \mathbb{D}^2$ and some $R=(r_1,r_2)$ with $r_j=r_j(d,z^0) \in(\eta (d),d),$ $j\in\{1,2\},$ and certain $k^0=k^0(d,z^0)\leq p$ the polynomial $p_{k^0}$ is the main polynomial in the series \eqref{th7main} on $\mathbb{T}^2(z^0,\frac{R}{\mathbf{L}(z^0)})$ 
 then the analytic function $F$ in $\mathbb{D}^2$  has bounded  $\mathbf{L}$-index in joint variables.
\end{theorem}

\begin{remark}
Theorem \ref{cor1} and \ref{petr8} are new even for entire functions of bounded $\mathbf{L}$-index in joint variables. 
 \end{remark}

{\footnotesize

 \vsk

 Department of Higher Mathematics
 
 Ivano-Frankivs'k National Technical University of Oil and Gas
 
 andriykopanytsia@gmail.com
 
 Department of Function Theory and Theory of Probability
 
 Ivan Franko National University of Lviv
 
 petrechko.n@gmail.com, olskask@gmail.com

}
\vs

\received{?} 

\end{document}